\documentclass{amsart}
\usepackage{amsmath, amsfonts, amssymb}
\usepackage{graphicx}
\usepackage{psfrag}

\newtheorem{mainthm}{Theorem}
\newtheorem{thm}{Theorem}[section] 
\newtheorem{lemma}[thm]{Lemma}
\newtheorem{prop}[thm]{Proposition}
\newtheorem{cor}[thm]{Corollary}

\theoremstyle{definition}

\newtheorem{example}[thm]{Example}
\newtheorem{defn}[thm]{Definition}

\DeclareMathOperator{\Fl}{Fl}

\DeclareMathOperator{\GL}{GL}
\DeclareMathOperator{\SL}{SL}

\newcommand{\bP}{{\mathbb P}}

\newcommand{\N}{{\mathbb N}}
\newcommand{\Z}{{\mathbb Z}}

\newcommand{\C}{{\mathbb C}}
\newcommand{\cA}{{\mathcal A}}
\newcommand{\cB}{{\mathcal B}}

\newcommand{\cR}{{\mathcal R}}

\newcommand{\wt}{\widetilde}
\newcommand{\wh}{\widehat}

\newcommand{\ignore}[1]{}

\newcommand{\xra}{\xrightarrow}

\newcommand{\pic}[2]{\includegraphics[scale=#1]{#2}}

\begin{document}

\title[The puzzle conjecture for two-step flag manifolds]{The puzzle
  conjecture for the cohomology of two-step flag manifolds}

\date{March 31, 2016}

\author[A. S. Buch]{Anders Skovsted Buch}
\address{Department of Mathematics, Rutgers University, 110
  Frelinghuysen Road, Piscataway, NJ 08854, USA}
\email{asbuch@math.rutgers.edu}

\author[A. Kresch]{Andrew Kresch}
\address{Institut f\"ur Mathematik,
Universit\"at Z\"urich, Winterthurerstrasse 190,
CH-8057 Z\"urich, Switzerland}
\email{andrew.kresch@math.uzh.ch} 

\author[K. Purbhoo]{Kevin Purbhoo}
\address{Combinatorics \& Optimization Department, University of Waterloo
200 University Ave W, Waterloo, Ontario, Canada N2L 3G1}
\email{kpurbhoo@math.uwaterloo.ca}

\author[H. Tamvakis]{Harry Tamvakis}
\address{Department of Mathematics, University of Maryland, 1301 Mathematics
Building, College Park, MD 20742, USA}
\email{harryt@math.umd.edu}

\thanks{The authors were supported in part by NSF grants DMS-0906148
  and DMS-1205351 (Buch), the Swiss National Science Foundation
  (Kresch), an NSERC discovery grant (Purbhoo), and NSF grants 
  DMS-0901341 and DMS-1303352 (Tamvakis).}

\begin{abstract}
  We prove a conjecture of Knutson asserting that the Schubert
  structure constants of the cohomology ring of a two-step flag
  variety are equal to the number of puzzles with specified border
  labels that can be created using a list of eight puzzle pieces.  
  As a consequence, we obtain a puzzle formula for the Gromov-Witten
  invariants defining the small quantum cohomology ring of a Grassmann
  variety of type A.  
  The proof of the conjecture proceeds by showing that the puzzle 
  formula defines an associative product on the cohomology ring of the
  two-step flag variety.  
  It is based on an explicit bijection of
  gashed puzzles that is analogous to the jeu de taquin algorithm but
  more complicated.
\end{abstract}


\maketitle

\section{Introduction}\label{sec:intro}

At the end of the last millennium, Knutson gave an elegant conjecture 
for the Schubert structure constants 
of the cohomology ring
of any partial flag variety $\SL(n)/P$ of type A \cite{knutson:conjectural}.  
The conjecture states that each
Schubert structure constant in $H^*(\SL(n)/P; \Z)$ is equal to the number
of triangular puzzles with specified border labels that can be created
using a list of puzzle pieces. The special case for Grassmannians was
established in 
\cite{knutson.tao.woodward:honeycomb, knutson.tao:puzzles}.
Unfortunately, Knutson quickly discovered
counterexamples to his general conjecture. Buch, Kresch, and Tamvakis
later proved that the Gromov-Witten invariants defining the small
quantum cohomology ring of a Grassmann variety are equal to Schubert
structure constants on two-step flag varieties, and it was suggested
that Knutson's conjecture might be true in this important special
case \cite{buch.kresch.tamvakis:gromov-witten}.
This was supported by verifying with the help of a computer that the
conjecture is correct for all two-step varieties 
$\Fl(a,b;n)$ with $n \leq 16$. 
The purpose of the present paper is to give a proof of Knutson's
conjecture for arbitrary two-step flag varieties.
As a consequence, we obtain a {\em quantum Littlewood-Richardson rule}: 
a puzzle formula for the three point, genus
zero Gromov-Witten invariants on any Grassmannian of type A,
see section~\ref{sec:qlrrule}.

After Knutson formulated his general conjecture and the 
work~\cite{buch.kresch.tamvakis:gromov-witten} appeared,
a different positive formula for the Schubert structure
constants on two-step flag varieties was proved by Coskun
\cite{coskun:littlewood-richardson}.
This rule expresses the structure constants as the number of
certain chains of diagrams called {\em mondrian tableaux}, 
which correspond
to the components of a degeneration of an intersection of Schubert
varieties.  Although it is well adapted to the geometry, Coskun's
rule does not address the validity of Knutson's conjecture for two-step
flag varieties.

Recent work of Knutson and Purbhoo \cite{knutson.purbhoo:product}
shows that the Belkale-Kumar coefficients
\cite{belkale.kumar:eigenvalue} for $\SL(n)/P$ are computed by a
special case of Knutson's original puzzle conjecture.  This special
case uses only a subset of the puzzle pieces, and it is not limited to
2-step flag varieties.  On the other hand, the set of Belkale-Kumar
coefficients is a proper subset of the structure constants of the
cohomology ring of any flag variety other than a Grassmannian.

In this paper, a {\em puzzle piece\/} means a (small) triangle from
the following list.
\[
\pic{.7}{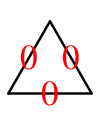} \ \ \pic{.7}{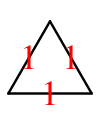} \ \ 
\pic{.7}{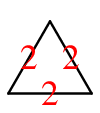} \ \ \pic{.7}{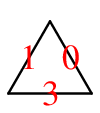} \ \ 
\pic{.7}{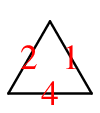} \ \ \pic{.7}{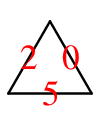} \ \ 
\pic{.7}{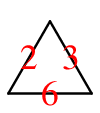} \ \ \pic{.7}{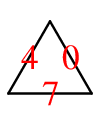}
\]
In Knutson's original conjecture the side labels of the puzzle pieces
were parenthesized strings of the integers 0, 1, and 2.  The labels
that are greater than two can be translated to such strings as
follows:
\[
\text{$3 = 10$, \ $4 = 21$, \ $5 = 20$, \ $6= 2(10)$, \ and \ $7 = (21)0$.}
\]
This description adds intuition to the puzzle pieces.  However, we
will stick to labels in the set $\{0,1,2,3,4,5,6,7\}$ in this paper.
The labels $0,1,2$ are called {\em simple\/} and the other labels
$3,4,5,6,7$ are called {\em composed}.

A {\em triangular puzzle\/} is an equilateral triangle made from
puzzle pieces with matching labels.  The puzzle pieces may be rotated
but not reflected.  If all labels on the border of a puzzle are
simple, then all composed labels in the puzzle are uniquely determined
from the simple labels.  One may therefore omit the edges with
composed labels in pictures of puzzles.  The following two pictures
show the same puzzle with and without its composed labels.
\[
\pic{.6}{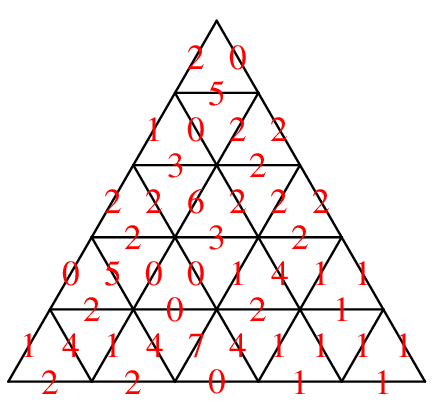} \ \ \ \ \ \ \ \ \pic{.6}{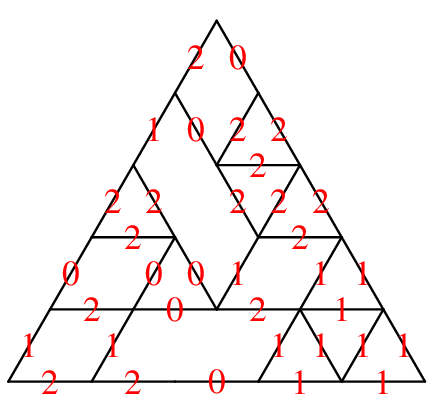}
\]

A puzzle with simple border labels is the same as an equilateral
triangle made of {\em composed puzzle pieces\/} from the following
list.  All can be rotated and the fourth and sixth can be stretched
in the direction of the longest side:
\[
\pic{.6}{st0} \ \ \pic{.6}{st1} \ \ \pic{.6}{st2} \ \ \pic{.6}{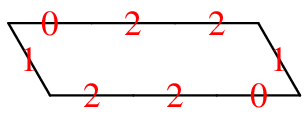}
\ \ \pic{.6}{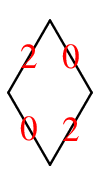} \ \ \pic{.6}{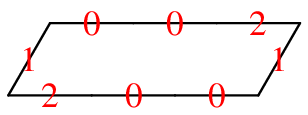}
\]

The Schubert varieties $X_u$ of the flag manifold $X = \Fl(a,b;n)$ are
indexed by integer vectors $u$ of length $n$, with $a$ entries equal
to 0, $b-a$ entries equal to 1, and $n-b$ entries equal to 2, see
section~\ref{sec:strategy}.  Such vectors will be called {\em
  012-strings\/} for $X$.  Knutson's conjecture for two-step flag
varieties is the following result.

\begin{mainthm}\label{thm:main}
  Let $X_u$, $X_v$, and $X_w$ be Schubert varieties in $\Fl(a,b;n)$.
  Then the triple intersection number
  \[
  \int_{\Fl(a,b;n)} [X_u] \cdot [X_v] \cdot [X_w]
  \]
  is equal to the number of triangular puzzles for which $u$, $v$, and
  $w$ are the labels on the left, right, and bottom sides, in
  clockwise order.
  \[
  \psfrag{u}{\!\!\!\!\!$u$}
  \psfrag{v}{\ \ $v$}
  \psfrag{w}{\raisebox{-3mm}{$w$}}
  \pic{.5}{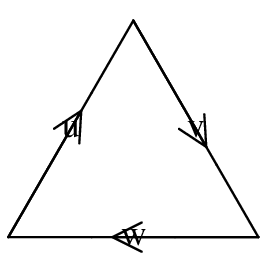}\vspace{2mm}
  \]
\end{mainthm}

Theorem~\ref{thm:main} is proved by establishing that the puzzle rule
defines an associative product on the cohomology ring of $X$, following
the strategy introduced in \cite{buch.kresch.tamvakis:littlewood-richardson}
in the case of Grassmannians.  
The corresponding identities among structure constants are obtained from
an explicit bijection of puzzles.  This bijection is a generalization
of the classical jeu de taquin algorithm on semistandard Young
tableaux (see \cite{purbhoo:mosaics} for a discussion of the bijections
between tableaux and puzzles),
but is significantly more involved and is defined by a list
of 80 different 
rules for propagating a {\em gash\/} from one side of
a puzzle to another.  A gash involves a pair of puzzle edges where the
puzzle pieces do not have matching labels, a notion introduced by
Knutson and Tao \cite{knutson.tao:puzzles}, and extended to our setting
in section~\ref{sec:gashes}.  The large number of 
propagation rules makes a few parts of our proof tedious but still 
straightforward to verify.

One important property of
the jeu de taquin algorithm for tableaux is that consecutive 
Sch\"utzenberger sliding
paths do not cross each other (see e.g.
\cite[section 2]{buch.kresch.tamvakis:littlewood-richardson}).  
This implies that a
suitable bijection of tableaux can be obtained by repeating the jeu
de taquin algorithm several times, an idea which has been used in many
proofs of the classical Littlewood-Richardson rule for Grassmannians,
see e.g.\ \cite{fulton:young, leeuwen:littlewood-richardson,
  buch.kresch.tamvakis:littlewood-richardson} and the references therein.
Unfortunately, this property does not hold for the corresponding {\em
  propagation paths\/} in puzzles for two-step flags, 
which may in fact cross each other. Example~\ref{example:cross} 
illustrates the problem.
We overcome this difficulty by carrying out several propagations
{\em simultaneously\/}, interlacing the individual steps, to obtain 
the desired bijection.  Naturally, this requires some care, and
the precise manner in which the simultaneous propagations are controlled 
is the main technical innovation in this paper.  The issue of crossing
propagation paths shapes our proof in subtle but significant ways;
for example, while there are many possible ways to formulate the 
propagation rules, our presentation is tailored to handle crossings
as seamlessly as possible.

This paper is organized as follows.  In section \ref{sec:strategy} we
reduce the proof of Theorem~\ref{thm:main} to the verification of two
identities which roughly state that the puzzle rule defines an
associative ring.  Section~\ref{sec:multone} proves one of these
identities, which says that multiplication by one has the expected
result.  We also reformulate the puzzle rule in terms of rhombus
shaped puzzles and identify the required properties of a bijection on
such puzzles.  Section~\ref{sec:pieri} defines a relation on strings
of puzzle labels that generalizes the Pieri rule for two-step flag
varieties.  This relation is crucial for carrying out multiple
propagations simultaneously and for dealing correctly with crossing
propagation paths.  In section~\ref{sec:gashes} we give an informal
discussion of propagations in two-step puzzles, after which
section~\ref{sec:propagation} gives the complete list of 
propagation rules together with case-by-case analysis that verifies
that every (unfinished) gash can be moved by a unique rule.  
At a first reading, section~\ref{sec:srA} through \ref{sec:srF} may
very well be skimmed, with greater attention given to section 
\ref{sec:properties}, which records properties essential to the proof 
of the main result.
Section~\ref{sec:general} puts the combinatorial constructions
together to obtain the required bijection and finish the proof.
Finally, section~\ref{sec:qlrrule} applies Theorem~\ref{thm:main} to 
obtain a quantum Littlewood-Richardson rule for the Gromov-Witten
invariants on Grassmannians.

\section{Strategy of the proof}\label{sec:strategy}

Theorem~\ref{thm:main} will be proved by applying the following
principle to the multiplicative action of the cohomology ring
$H^*(X;\Z)$ on itself.  This principle was first applied to classical
Schubert
calculus in \cite{buch.kresch.tamvakis:littlewood-richardson}.  A related
principle in the setting of equivariant cohomology was introduced in
\cite{molev.sagan:littlewood-richardson} and used in
\cite{knutson.tao:puzzles}.

\begin{lemma}\label{lem:principle}
  Let $R$ be an associative ring with unit $1$, let $S \subset R$ be a
  subset that generates $R$ as a $\Z$-algebra, and let $M$ be a left
  $R$-module.  Let $\mu : R \times M \to M$ be any $\Z$-bilinear map
  satisfying that, for all $r \in R$, $s \in S$, and $m \in M$ we have
  $\mu(1,m) = m$ and $\mu(rs,m) = \mu(r,sm)$.  Then $\mu(r,m) = rm$
  for all $(r,m) \in R \times M$.
\end{lemma}
\begin{proof}
  Since $R$ is generated by $S$ and $\mu$ is linear in its first
  argument, it is enough to show that $\mu(r,m) = rm$ whenever $m \in
  M$ and $r = s_1 s_2 \cdots s_k$ is a product of factors $s_i \in S$.
  This follows from the assumptions by induction on $k$.
\end{proof}

Let $X = \Fl(a,b;n) = \{(A,B) : A \subset B \subset \C^n \text{ and }
\dim(A)=a \text{ and } \dim(B)=b \}$ be the variety of two-step flags
in $\C^n$ of dimensions $(a,b)$.  Let $e_1,\dots,e_n$ be the standard
basis for $\C^n$.  For a subset $S \subset \C^n$, we let $\langle S
\rangle \subset \C^n$ denote the span of $S$.  A 012-string for $X$ is
a string $u = (u_1,\dots,u_n)$ with $a$ zeros, $b-a$ ones, and $n-b$
twos; this corresponds to a minimal length coset representative
for the parabolic subgroup $S_a \times S_{b-a} \times S_{n-b}$
of the Weyl group $S_n$.

Given a 012-string for $X$, consider the point $(A_u,B_u) \in
X$, where $A_u = \langle e_i : u_i=0 \rangle$ and $B_u = \langle e_i :
u_i \leq 1 \rangle$.  The Schubert variety $X_u \subset X$ is the
closure of the orbit of $(A_u,B_u)$ for the action of the lower
triangular matrices in $\GL(n)$.  Equivalently, $X_u$ is the variety
of points $(A,B) \in X$ for which $\dim(A \cap \langle e_p,\dots,e_n
\rangle) \geq \dim(A_u \cap \langle e_p,\dots,e_n \rangle)$ and
$\dim(B \cap \langle e_p,\dots,e_n \rangle) \geq \dim(B_u \cap \langle
e_p,\dots,e_n \rangle)$ for all $p \in [1,n]$.  The codimension of
$X_u$ in $X$ is equal to the number of inversions $\ell(u) = \# \{
(i,j) : 1 \leq i < j \leq n \text{ and } u_i > u_j \}$.  The Schubert
classes $[X_u]$ given by all $012$-strings for $X$ form a basis for
the cohomology ring $H^*(X;\Z)$.  The Poincar\'e dual of the 012-string
$u = (u_1,u_2,\dots,u_n)$ is the reverse string $u^\vee =
(u_n,u_{n-1},\dots,u_1)$.  With this notation we have $\int_X [X_u]
\cdot [X_v] = \delta_{u^\vee,v}$.

Given 012-strings $u$, $v$, and $w$ for $X$, we let $C^w_{u,v}$ be the
number of triangular puzzles with labels $u$, $v$, and $w$ on the
left, right, and bottom borders, with $u$ and $v$ in clockwise
direction and $w$ in counter-clockwise direction.
\[
\psfrag{u}{\!\!\!\!\!$u$}
\psfrag{v}{\ \ $v$}
\psfrag{w}{\raisebox{-3mm}{$w$}}
\pic{.5}{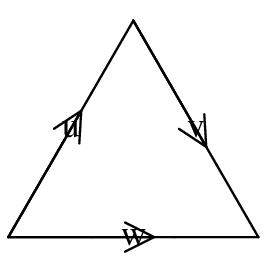}\smallskip
\]
Define a $\Z$-bilinear map $\mu : H^*(X;\Z) \times H^*(X;\Z) \to
H^*(X;\Z)$ by
\[ 
\mu([X_u], [X_v]) = \sum_w C^w_{u,v} [X_w] \,,
\]
where the sum is over all 012-strings $w$ for $X$.  It follows from
Poincar\'e duality that Theorem~\ref{thm:main} is equivalent to the
identity
\begin{equation}\label{eqn:lrrule}
  \mu([X_u], [X_v]) = [X_u] \cdot [X_v]
\end{equation}
for all 012-strings $u$ and $v$ for $X$.  We will prove this by
identifying a generating subset $S \subset H^*(X;\Z)$ of special
Schubert classes that satisfies the conditions of
Lemma~\ref{lem:principle}.

Given two 012-strings $u$ and $u'$, with $u = (u_1,u_2,\dots,u_n)$, we
write $u \xrightarrow{1} u'$ if there exist indices $i<j$ such that
(1) $u_i \in \{0,1\}$, (2) $u_j = 2$, (3) $u_k < u_i$ for all $k$ with
$i < k < j$, and (4) $u'$ is obtained from $u$ by interchanging $u_i$
and $u_j$.  This corresponds to a covering relation in the
ordering induced from the Bruhat order on $S_n$.
More generally, for $p \in \N$ we write $u \xrightarrow{p}
u'$ if there exists a sequence $u = u^0 \xrightarrow{1} u^1
\xrightarrow{1} \cdots \xrightarrow{1} u^p = u'$
such that, if $u^t$ is obtained from $u^{t-1}$ by interchanging the
entries of index $i_t$ and $j_t$, then $j_t \leq i_{t+1}$ for all $t
\in [1,p-1]$.  For example, the chain
\[
\begin{split}
&
(0, 2, 1, 1, 0, 0, 2, 0, 2, 0, 2) \xra{\,1\,}
(2, 0, 1, 1, 0, 0, 2, 0, 2, 0, 2) \xra{\,1\,}
(2, 0, 1, 2, 0, 0, 1, 0, 2, 0, 2) \\
& \ \ \ \ \ \ \xra{\,1\,}
(2, 0, 1, 2, 0, 0, 2, 0, 1, 0, 2) \xra{\,1\,}
(2, 0, 1, 2, 0, 0, 2, 0, 1, 2, 0) \,.
\end{split}
\]
implies that
\[
(0, 2, 1, 1, 0, 0, 2, 0, 2, 0, 2) \xra{\,4\,}
(2, 0, 1, 2, 0, 0, 2, 0, 1, 2, 0) \,.
\]
Given an integer $p$ with $0 \leq p \leq n-b$, we let $\mathbf{p}$
denote the 012-string $\mathbf{p} = (0^a, 1^{b-a-1}, 2^p, 1,
2^{n-b-p})$.  This string defines the special Schubert variety
\[
X_{\mathbf p} = \{ (A,B) \in X \mid B \cap \langle e_{b+p},\dots,e_n \rangle
\neq 0 \} \,.
\]
The corresponding special Schubert class is the Chern class
$[X_{\mathbf p}] = c_p(\C^n_X/\cB)$, where $\cA \subset \cB \subset
  \C^n_X = \C^n\times X$ is the tautological flag of subbundles on
  $X$.  The Pieri formula for $X$ states that
  \cite{lascoux.schutzenberger:polynomes*1, sottile:pieris}
\begin{equation}\label{eqn:pieri}
  [X_{\mathbf p}] \cdot [X_u] = \sum_{u \xra{p} u'} [X_{u'}] \,,
\end{equation}
where the sum is over all 012-strings $u'$ for which $u \xra{p} u'$.

We will derive (1) as a consequence of the following two identities,
which will be proved later using an explicit bijection of puzzles
(c.f. \cite[Proposition 1]{buch.kresch.tamvakis:littlewood-richardson}).
Recall that $\mathbf{0} = (0^a, 1^{b-a}, 2^{n-b})$ is the identity
012-string for which $[X_{\mathbf 0}] = 1 \in H^*(X;\Z)$.

\begin{prop}\label{prop:key}
  Let $u$, $v$, and $w$ be 012-strings for $X$ and let $p \geq 0$ be
  an integer.  Then we have
  \begin{equation}\label{eqn:id}
    C^w_{\mathbf{0},u} = \delta_{u,w}
  \end{equation}
  and
  \begin{equation}\label{eqn:assoc}
    \sum_{u \xra{p} u'} C^w_{u',v} = \sum_{v \xra{p} v'} C^w_{u,v'}
    \,.
  \end{equation}
\end{prop}

Fix an orthogonal form on $\C^n$ and set $\wh X = \Fl(n-b,n-a;n)$.
Then the map $\phi : X \to \wh X$ that sends a point $(A,B)$ to
$(B^\perp,A^\perp)$ is an isomorphism of varieties, called the {\em
  duality isomorphism}.  The corresponding isomorphism of cohomology
rings $\phi^* : H^*(\wh X; \Z) \to H^*(X;\Z)$ is given by $\phi^* [\wh
X_{\wh u}] = [X_u]$, where $\wh u = (2-u_n,2-u_{n-1},\dots,2-u_1)$.

For each integer $p \in [0,a]$, define the 012-string $\wt{\mathbf p}
= (0^{a-p}, 1, 0^p, 1^{b-a-1}, 2^{n-b})$.  This string defines the
special Schubert variety
\[
X_{\wt{\mathbf p}} = \{ (A,B) \in X \mid \dim(A \cap \langle
e_{a-p+2},\dots,e_n \rangle) \geq p \}
\]
and the special Schubert class $[X_{\wt{\mathbf p}}] = c_p(\cA^\vee) \in
H^*(X;\Z)$.  By applying $\phi^*$ to both sides of the Pieri rule
(\ref{eqn:pieri}) for $\wh X$, we obtain the identity
\begin{equation}\label{eqn:pieri2}
  [X_{\wt{\mathbf p}}] \cdot [X_u] 
  = \sum_{u \xra{\wt p} u'} [X_{u'}] 
\end{equation}
in $H^*(X;\Z)$, where we write $u \xra{\wt p} v$ if and only if $\wh u
\xra{p} \wh v$.

The duality isomorphism has a corresponding bijection on puzzles that
reflects a puzzle in a vertical line and substitutes all labels
according to the following rule:
\[
0 \mapsto 2 \ ; \ \ 
1 \mapsto 1 \ ; \ \ 
2 \mapsto 0 \ ; \ \ 
3 \mapsto 4 \ ; \ \ 
4 \mapsto 3 \ ; \ \ 
5 \mapsto 5 \ ; \ \ 
6 \mapsto 7 \ ; \ \ 
7 \mapsto 6 \,.
\]
\[
\pic{.5}{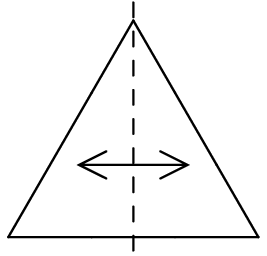}
\]
This bijection implies that we have $C^w_{u,v} = C^{\wh w}_{\wh v,\wh
  u}$ for all 012-strings $u,v,w$ for $X$.  In particular,
equation (\ref{eqn:assoc}) applied to $\wh X$ gives the identity
\begin{equation}\label{eqn:assoc2}
\sum_{u \xra{\wt p} u'} C^w_{u',v} = \sum_{v \xra{\wt p} v'}
C^w_{u,v'} \,,
\end{equation}
for all 012-strings $u,v,w$ for $X$ and integers $p \in [0,a]$.

\begin{proof}[Proof of Theorem~\ref{thm:main}]
  Set $R = M = H^*(X;\Z)$ and 
  \[
  S = \{ [X_{\mathbf p}] : 1 \leq p \leq n-b \}
  \cup \{ [X_{\wt{\mathbf p}}] : 1 \leq p \leq a \}\,.  
  \]
  By using that $X$ is a
  Grassmann bundle over a Grassmann variety it follows that $R$ is
  generated by $S$.  The identity (\ref{eqn:id}) shows that
  $\mu(1,[X_u]) = [X_u]$, and (\ref{eqn:assoc}) and (\ref{eqn:assoc2})
  together with the Pieri formulas (\ref{eqn:pieri}) and
  (\ref{eqn:pieri2}) imply 
\[
   \mu([X_u] \cdot [X_{\mathbf p}], [X_v])
  = \sum_{u \xra{p} u'} \mu([X_{u'}], [X_v])
  = \sum_{v \xra{p} v'} \mu([X_u], [X_{v'}])
  = \mu([X_u],  [X_{\mathbf p}] \cdot [X_v])
\]
for all 012-strings $u$ and $v$ for $X$ and 
$1 \leq p \leq n-b$, and an analogous identity with
$[X_{\wt{\mathbf p}}]$ for $1 \leq p \leq a$.
By the bilinearity of $\mu$ this shows that the conditions of
  Lemma~\ref{lem:principle} are satisfied.  We deduce that 
  equation~\eqref{eqn:lrrule} holds
   for all 012-strings $u$ and $v$,
  as required.
\end{proof}

\section{Multiplication by one}\label{sec:multone}

In this section we prove the first claim in Proposition~\ref{prop:key}
and use it to reformulate the puzzle formula in terms of
rhombus-shaped puzzles.  We need the following lemma.

\begin{lemma}\label{lem:idrow}
  Let ${\mathbf 0} = (0^{n_0},1^{n_1},2^{n_2})$ be an identity string
  and let $x \in \{0,1,2\}$ be a simple label that occurs in this
  string.  Then there exists a unique union of matching puzzle pieces
  of the form
  \[
  \psfrag{0}{\raisebox{-3mm}{${\mathbf 0}$}}
  \psfrag{1}{\raisebox{4mm}{${\mathbf 0'}$}}
  \psfrag{x}{\raisebox{.7mm}{$\!\!\!\!x$}}
  \psfrag{a}{\raisebox{.7mm}{$\ \ x$}}
  \pic{.75}{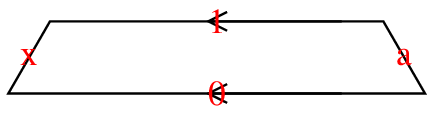}
  \]
  with left label $x$ and bottom labels ${\mathbf 0}$ in right-to-left
  direction.  The right border of this unique puzzle has label $x$,
  and the labels on the top border is the identity string ${\mathbf
    0'}$ obtained by removing one copy of $x$ from ${\mathbf 0}$.
\end{lemma}
\begin{proof}
  We consider each possible value of $x$ in turn.  If $x=0$, then the
  shape must be filled with (unions of) puzzle pieces from the
  following list:
  \[
  \raisebox{7mm}{\text{(a)}} \ \pic{.75}{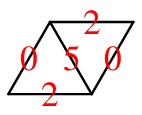} \hspace{7mm}
  \raisebox{7mm}{\text{(b)}} \ \pic{.75}{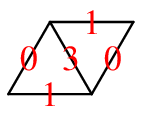} \hspace{7mm}
  \raisebox{7mm}{\text{(c)}} \ \pic{.75}{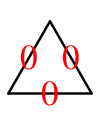}   \hspace{7mm}
  \raisebox{7mm}{\text{(d)}} \ \pic{.75}{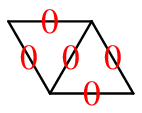}
  \]
  In fact, if we fill the shape from left to right, then we are forced
  to place the rhombus (a) above each 2 on the bottom border.  After
  this we must place the rhombus (b) above each 1 on the bottom
  border; the only alternative rhombus
  \raisebox{-4mm}{\pic{.5}{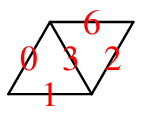}} cannot be used because each 1-label
  on the bottom border is followed by a 0 or a 1 to the right.
  Finally, the triangle (c) must be placed above the first 0 on the
  bottom border, and the rest of the shape must be filled with the
  rhombus (d).

  If $x=1$, then a similar argument shows that the shape must be
  filled with the (unions of) puzzle pieces:
  \[
  \pic{.75}{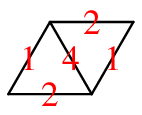} \hspace{7mm}
  \pic{.75}{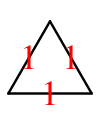} \hspace{7mm}
  \pic{.75}{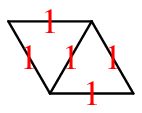} \hspace{7mm}
  \pic{.75}{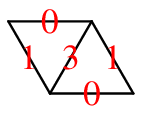}
  \]
  And if $x=2$, then the shape must be filled with the pieces:
  \[
  \pic{.75}{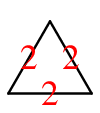} \hspace{7mm}
  \pic{.75}{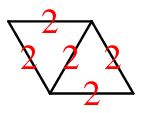} \hspace{7mm}
  \pic{.75}{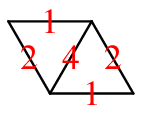} \hspace{7mm}
  \pic{.75}{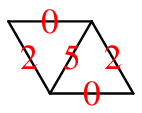}
  \]
  In all three cases exactly one single triangle is used, with the
  label $x$ on all sides.  This accounts for the removed $x$ on the
  top border.
\end{proof}

The first identity in Proposition~\ref{prop:key} follows from the
following corollary.

\begin{cor}\label{cor:idpuz}
  Let $v$ be any 012-string for $X$ and let ${\mathbf 0} =
  (0^a,1^{b-a},2^{n-b})$ be the identity string.  Then there exists a
  unique triangular puzzle with labels ${\mathbf 0}$ on the left
  border and labels $v$ on the right border, both in clockwise
  direction.
  \[
  {
    \psfrag{u}{\!\!\!\!\!$\mathbf 0$}
    \psfrag{v}{\ \ $v$}
    \psfrag{w}{\raisebox{-3mm}{$v$}}
    \pic{.5}{trilr}
  }
  \]
  The bottom labels of this unique puzzle are $v$, in counter-clockwise
  direction.
\end{cor}
\begin{proof}
  This follows by induction on the number of rows, using the 120
  degree clockwise rotation of Lemma~\ref{lem:idrow}.
\end{proof}

For technical reasons it is convenient to express the constants
$C^w_{u,v}$ in terms of puzzles of rhombus shape.  We will use this
interpretation in the proof of the second identity of
Proposition~\ref{prop:key}.

\begin{cor}\label{cor:rhombus}
  The constant $C^{w}_{u,v}$ is equal to the number of puzzles of the
  following rhombus shape, with top border $u$, right border ${\mathbf
    0}$, bottom border $v$, and left border $w$, with $u$, ${\mathbf
    0}$, and $v$ in clockwise direction and $w$ in counter-clockwise
  direction.
  \[
    \psfrag{l}{$\!\!\!\!\!\!\!w$}
    \psfrag{r}{$\ \ {\mathbf 0}$}
    \psfrag{t}{\raisebox{3mm}{$u$}}
    \psfrag{b}{\raisebox{-2.5mm}{$v$}}
    \pic{.6}{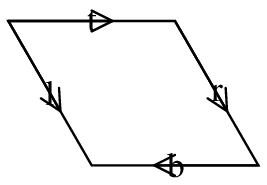}
  \]
\end{cor}
\begin{proof}
  Any such rhombus shaped puzzle consists of the (rotated) unique
  puzzle from Corollary~\ref{cor:idpuz} in the lower-right half and a
  (rotated) triangular puzzle with border labels $u$, $v$, and $w$ in
  the upper-left half.
\end{proof}

Let $u$, $v$, and $w$ be 012-strings for $X$ and let $p \in [0,n-b]$.
To prove the second identity of Proposition~\ref{prop:key}, it
suffices to construct a bijection between the set of rhombus shaped
puzzles with border labels $u',{\mathbf 0},v,w$ such that $u
\xrightarrow{p} u'$, and the set of rhombus shaped puzzles with border
labels $u,{\mathbf 0},v',w$ such that $v \xrightarrow{p} v'$.\medskip
\[
  {
    \psfrag{l}{$\!\!\!\!\!\!\!w$}
    \psfrag{r}{$\ \ {\mathbf 0}$}
    \psfrag{t}{\raisebox{2.7mm}{$\!\!\!u'$}}
    \psfrag{b}{\raisebox{-2.2mm}{$v$}}
    \pic{.6}{rhompuz}
  }
  \raisebox{8mm}{\ \ \ \ $\longleftrightarrow$\ \ \ \ }
  {
    \psfrag{l}{$\!\!\!\!\!\!\!w$}
    \psfrag{r}{$\ \ {\mathbf 0}$}
    \psfrag{t}{\raisebox{2.7mm}{$\!\!\!u$}}
    \psfrag{b}{\raisebox{-2.2mm}{$v'$}}
    \pic{.6}{rhompuz}
  }
\]
We will construct a more general bijection where the top and bottom
borders are not required to have simple labels.  The advantage of this
is that we can restrict our attention to puzzles with a single row.
The first ingredient in our construction is an appropriate
generalization of the Pieri relation $u \xra{p} v$ for strings of
arbitrary labels.  This is the subject of the next section.

\section{A Pieri rule for label strings}\label{sec:pieri}

Define a {\em label string\/} to be any finite sequence $u =
(u_1,u_2,\dots,u_\ell)$ of integers from the set $[0,7] =
\{0,1,2,3,4,5,6,7\}$.  These strings are generalizations of the
012-strings that represent Schubert classes on two-step flag
varieties.  In this section we introduce a generalization of the Pieri
relation $u \xra{p} v$ that has meaning when $u$ and $v$ are arbitrary
label strings of the same length.

We start by defining the basic relation $u \xra{1} v$.  This relation
implies that $v$ is obtained by changing exactly two entries of $u$.
There are 15 possible rules for how the entries can be changed, and in
each case there are restrictions on which entries can appear between
the entries being changed.  Each rule is determined by a triple
$((a_1,b_1), S, (a_2,b_2))$ where $a_1,b_1,a_2,b_2 \in [0,7]$ and $S
\subset [0,7]$.  The corresponding rule says that, if $u$ contains a
substring consisting of $a_1$ followed by any number of integers from
$S$ and ending in $a_2$, then one may replace $a_1$ in the substring
with $b_1$ and simultaneously replace $a_2$ with $b_2$.  We will use
the following graphical representation of the rule:
\[
\psfrag{1}{\raisebox{1.5mm}{$a_1$}}
\psfrag{5}{\raisebox{1.5mm}{$a_2$}}
\psfrag{2}{\raisebox{-1mm}{$b_1$}}
\psfrag{7}{\raisebox{-1mm}{$b_2$}}
\psfrag{03*}{\ \,\raisebox{.5mm}{$S*$}}
\pic{.75}{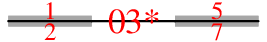}
\]
The set $S$ is specified by listing its elements followed by a star to
indicate that its elements can be repeated.  If $S$ is empty, then the
middle third of the line segment is omitted.  The complete list of
rules is given in Table~\ref{tab:pierirel}.  These rules are organized
into six {\em types\/} called A, B, C, D, E, F, (these have no relation
to the classification of types in Lie theory!).
Notice that just two of the rules relate $012$-strings, and these
reproduce the definition of $u \xra{1} v$ from section~\ref{sec:strategy};
the remaining rules follow a similar pattern.  As we will see in
section~\ref{sec:gashes}, the complete set of rules
defining $u \xra{1} v$
arises as a subset of the gashes that can occur in propagation algorithm.

\begin{defn}\label{def:pieri1}
  Let $u$ and $v$ be label strings of the same length.  Then the
  relation $u \xra{1} v$ holds if and only if $v$ can be obtained from
  $u$ by using one of the rules in Table~\ref{tab:pierirel}.  In this
  case we say that the relation $u \xra{1} v$ has {\em index\/}
  $(i,j)$, where $i<j$ are the unique integers such that $u_i \neq
  v_i$ and $u_j \neq v_j$.
\end{defn}

\begin{table}
\begin{tabular}{|c|c|}
\hline
Type & Rule \\
\hline\hline
A & \pic{.75}{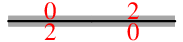} \\
\hline
B & \pic{.75}{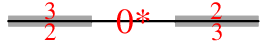} \\
\hline
C & \pic{.75}{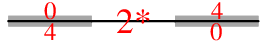} \\
\hline
D & \pic{.75}{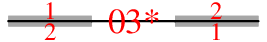} \\
  & \pic{.75}{gash_4} \\
  & \pic{.75}{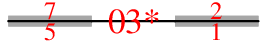} \\
  & \pic{.75}{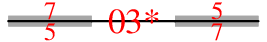} \\
\hline
E & \pic{.75}{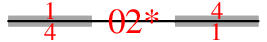} \\
  & \pic{.75}{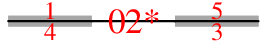} \\
  & \pic{.75}{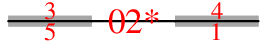} \\
  & \pic{.75}{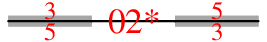} \\
\hline
F & \pic{.75}{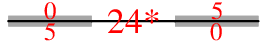} \\
  & \pic{.75}{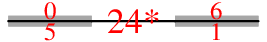} \\
  & \pic{.75}{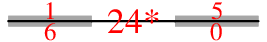} \\
  & \pic{.75}{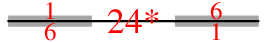} \\
\hline
\end{tabular}
\medskip
\caption{Rules for the Pieri relation on label strings.}
\label{tab:pierirel}
\end{table}

\begin{example}
  According to the second rule of type D, we have $72{\mathbf
    1}3033{\mathbf 5}644 \xra{1} 72{\mathbf 2}3033{\mathbf 7}644$, and
  this relation has index $(3,8)$.
\end{example}

\begin{defn}\label{def:pierip}
  Let $u$ and $v$ be label strings of the same length.  A {\em Pieri
    chain\/} from $u$ to $v$ is a sequence $u = u^0 \xra{1} u^1
  \xra{1} \cdots \xra{1} u^p = v$ such that, if $u^{t-1} \xra{1} u^t$
  has index $(i_t,j_t)$ for each $t$, then $i_s < j_t$ whenever $s
  \leq t$.  The Pieri chain is {\em right-increasing\/} if it
  satisfies the stronger condition $j_1 < j_2 < \dots < j_p$.  We will
  write $u \xra{p} v$ if there exists a Pieri chain of length $p$ from
  $u$ to $v$.
\end{defn}

\begin{example}\label{exm:rightpieri}
  We have $04730202245 \xra{5} 40720522015$, with (right-increasing)
  Pieri chain $04730202245 \xra{1} {\mathbf{40}}730202245 \xra{1}
  407{\mathbf 2}0{\mathbf 3}02245 \xra{1} 407203{\mathbf{20}}245
  \xra{1} 4072032{\mathbf{20}}45 \xra{1} 40720{\mathbf 5}220{\mathbf
    1}5$.
\end{example}

Notice that some swaps of integers in a Pieri chain can happen inside
others: for instance, in example~\ref{exm:rightpieri} the fourth
swap (at position $(8,9)$) happens inside the fifth (at position
$(6,10)$).  In section~\ref{sec:general}, 
this property will be utilized to allow propagation paths to
cross each other in a controlled way.  The two results below are
essential for this application.  Notice also that the definitions
imply that $u \xra{p} v$ if and only if $v^\vee \xra{p} u^\vee$, where
$u^\vee$ denotes the label string $u$ in reverse order.

\begin{lemma}\label{lem:chainorder}
  Let $u \xra{1} v \xra{1} w$ be a Pieri chain, where $u \xra{1} v$
  has index $(i,j)$ and $v \xra{1} w$ has index $(k,l)$.  Assume that
  $k < j$.  Then we have either $i < k < l < j$, $u \xra{1} v$ has
  type E, and $v \xra{1} w$ has type A, or $k < i < j < l$, $u \xra{1}
  v$ has type A, and $v \xra{1} w$ has type E.  Furthermore, there
  exists a unique label string $v'$ such that $u \xra{1} v' \xra{1} w$
  is a Pieri chain with the inequalities and types interchanged.
\end{lemma}
\begin{proof}
  Assume that $u \xra{1} v$ follows the rule $((a_1,b_1),S,(a_2,b_2))$
  and that $v \xra{1} w$ follows the rule $((c_1,d_1),T,(c_2,d_2))$,
  both of which come from Table~\ref{tab:pierirel}:
  \[
  {
    \psfrag{1}{\raisebox{1.5mm}{$a_1$}}
    \psfrag{5}{\raisebox{1.5mm}{$a_2$}}
    \psfrag{2}{\raisebox{-1mm}{$b_1$}}
    \psfrag{7}{\raisebox{-1mm}{$b_2$}}
    \psfrag{03*}{\ \,\raisebox{.5mm}{$S*$}}
    \pic{.75}{gash_4}
  }
  \text{\ \ \ \ \ and\ \ \ \ \ }
  {
    \psfrag{1}{\raisebox{1.5mm}{$c_1$}}
    \psfrag{5}{\raisebox{1.5mm}{$c_2$}}
    \psfrag{2}{\raisebox{-1mm}{$d_1$}}
    \psfrag{7}{\raisebox{-1mm}{$d_2$}}
    \psfrag{03*}{\ \,\raisebox{.5mm}{$T*$}}
    \pic{.75}{gash_4}
  }
  \]
  Then we have $v_i=b_1$, $v_j=b_2$, $v_k=c_1$, $v_l=c_2$, $v_s\in S$
  for $i<s<j$, and $v_s\in T$ for $k<s<l$.  By inspection of
  Table~\ref{tab:pierirel} we have $b_1,c_2 \in \{2,4,5,6\}$, $b_2,c_1
  \in \{0,1,3,7\}$, and $S \cup T \subset \{0,2,3,4\}$.  It follows
  that $i,j,k,l$ are pairwise distinct integers, and the inequalities
  $i<j$, $k<l$, $i<l$, $k<j$ allow exactly four possibilities for
  their relative orderings.  We consider these possibilities in turn.

  Case 1: Assume that $i<k<j<l$.  Then $c_1\in S$ and $b_2 \in T$.
  Using that $b_2,c_1 \in \{0,1,3,7\} \cap \{0,2,3,4\} = \{0,3\}$, we
  deduce that both of the applied rules are not of type A, C, D, or F.
  This in turn implies that $c_1=b_2=3$ and $S\cup T \subset \{0,2\}$,
  a contradiction.

  Case 2: Assume that $k<i<l<j$.  This case is impossible by an
  argument similar to Case 1.


  Cases 3 and 4: Assume that $i<k<l<j$ or $k<i<j<l$.  Then $c_1,c_2
  \in S$ or $b_1,b_2 \in T$, which implies that the types of the rules
  are as stated in the lemma.  In both cases the label string $v'$ is
  obtained by setting $v'_k=d_1$, $v'_l=d_2$, and $v'_s=u_s$ for $s
  \neq k,l$.
\end{proof}

\begin{cor}\label{cor:label012}
  Let $u$ and $v$ be label strings of the same length and let $p \in
  \N$.

  \begin{itemize}
  \item[(a)] Assume that $u \xra{p} v$.  Then $u$ is a 012-string if
    and only if $v$ is a 012-string.

  \item[(b)] Assume that both $u$ and $v$ are 012-strings.  Then the
    relation $u \xra{p} v$ of the Pieri rule (\ref{eqn:pieri}) holds
    if and only if $u \xra{p} v$ holds in the sense of label strings.
  \end{itemize}
\end{cor}
\begin{proof}
  Only the rule of type A and the first rule of type D in
  Table~\ref{tab:pierirel} can be applied to a 012-string, and they
  will replace such a string with a new 012-string.  Part (a) and the
  special case of (b) in which $p=1$ follow from this.  Let $u = u^0
  \xra{1} u^1 \xra{1} \cdots \xra{1} u^p = v$ be a Pieri chain such
  that $u^{t-1} \xra{1} u^t$ has index $(i_t,j_t)$ for each $t$.
  Since no relation $u^{t-1} \xra{1} u^t$ has type E, it follows from
  Lemma~\ref{lem:chainorder} that $j_t \leq i_{t+1}$ for each $t$.
  The general case of part (b) follows from this.
\end{proof}

\begin{prop}\label{prop:rightincr}\mbox{}

  \begin{itemize}
  \item[(a)] Let $u$ and $v$ be label strings with $u \xra{p} v$.
    Then there exists a unique right-increasing Pieri chain from $u$
    to $v$.
    
  \item[(b)] Let $u = u^0 \xra{1} u^1 \xra{1} \cdots \xra{1} u^p = v$
    be any Pieri chain from $u$ to $v$, and let $(i_t,j_t)$ be the
    index of $u^{t-1} \xra{1} u^t$ for each $t \in [1,p]$.  If $j_1 =
    \min\{j_1,\dots,j_p\}$, then $u^1$ belongs to the unique
    right-increasing Pieri chain from $u$ to $v$.
  \end{itemize}
\end{prop}
\begin{proof}
  Let $u = u^0 \xra{1} u^1 \xra{1} \cdots \xra{1} u^p = v$ be any
  Pieri chain, and let $u^{t-1} \xra{1} u^t$ have index $(i_t,j_t)$.
  Assume that $j_{t+1} \leq j_t$  for some $t$.  Then $i_{t+1} < j_t$,
  so Lemma~\ref{lem:chainorder} implies that $i_t < i_{t+1} < j_{t+1}
  < j_t$.  Moreover, there exists a label string $v'$ such that
  $u^{t-1} \xra{1} v'$ has index $(i_{t+1},j_{t+1})$ and $v' \xra{1}
  u^{t+1}$ has index $(i_t,j_t)$.  By replacing $u^t$ with $v'$, we
  obtain a new Pieri chain where the pairs $(i_t,j_t)$ and
  $(i_{t+1},j_{t+1})$ are interchanged.  We repeat this procedure
  until $j_1<j_2<\dots<j_p$.  This shows that there exists at least
  one right-increasing Pieri chain from $u$ to $v$.  If the initial
  Pieri chain satisfies $j_1 = \min\{j_1,\dots,j_p\}$, then the string
  $u^1$ will never be replaced and will remain unchanged in the right
  increasing Pieri chain.

  To see that the right-increasing Pieri chain is uniquely determined
  from $u$ and $v$, assume that the last step $u^{p-1} \xra{1} v$
  follows the rule $((a_1,b_1),S,(a_2,b_2))$, and let $(i,j)$ be the
  index of the last step.  Notice that $j$ is the largest integer for
  which $u_j \neq v_j$, and we have $(a_2,b_2) = (u_j,v_j)$.  This
  pair determines the type of the rule, which in turn determines the
  set $S$.  Now $i$ is the largest integer for which $i<j$ and $v_i
  \notin S$.  We have $b_1 = v_i$, and the entire rule is determined
  by the triple $(b_1,a_2,b_2)$.  We now obtain $u^{p-1}$ by applying
  the inverse rule to $v$, and by induction there exists a unique
  right-increasing Pieri chain from $u$ to $u^{p-1}$.
\end{proof}

\section{Gashes and swap regions}\label{sec:gashes}

In this section we will work with parallelogram shaped puzzles with a single
row, such that the left and right border edges have simple labels.
Such puzzles will be called {\em single-row puzzles}.  We will say
that a single-row puzzle has {\em border\/} $(c_1,u,v,c_2)$ if $u$ is
the string of labels on the top border from left to right, $v$ is the
string of labels on the bottom border from left to right, $c_1$ is the
simple label on the left border, and $c_2$ is the simple label on the
right border.
\[
    \psfrag{l}{$\!\!\!\!\!c_1$}
    \psfrag{r}{\ \ $c_2$}
    \psfrag{t}{\raisebox{3.5mm}{$\!\!u$}}
    \psfrag{b}{\raisebox{-2.4mm}{$v$}}
    \pic{.6}{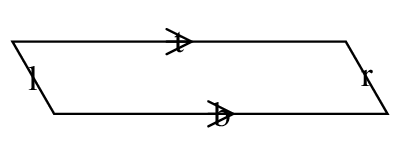}
\] 
%
%
Given label strings $u$ and $v'$ of the same length, simple labels
$c_1,c_2 \in \{0,1,2\}$, and an integer $p \geq 0$, we will construct
a bijection between the set of single-row puzzles with border
$(c_1,u',v',c_2)$ for which $u \xra{p} u'$, and the set of single-row
puzzles with border $(c_1,u,v,c_2)$ for which $v \xra{p} v'$.  We
start with the simplest case where $p=1$.

\begin{table}
\begin{tabular}{|c|c|c|c|}
\hline
Type & Left leg & Middle segment & Right leg \\
\hline
\hline
A & \pic{.75}{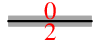} & & \pic{.75}{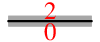} \\
\hline
B & \pic{.75}{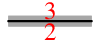} & \pic{.75}{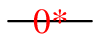}  & \pic{.75}{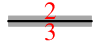} \\
\hline
C & \pic{.75}{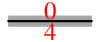} & \pic{.75}{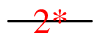}  & \pic{.75}{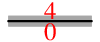} \\
  & \pic{.75}{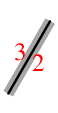} &                     & \pic{.75}{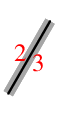} \\
\hline
D & \pic{.75}{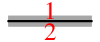} & \pic{.75}{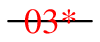} & \pic{.75}{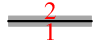} \\
  & \pic{.75}{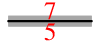} & \pic{.75}{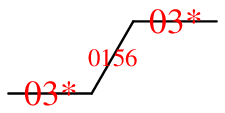} & \pic{.75}{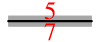} \\
  & \pic{.75}{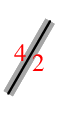} & \pic{.75}{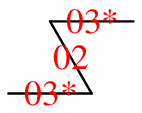} & \pic{.75}{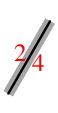} \\
\hline
E & \pic{.75}{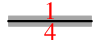} & \pic{.75}{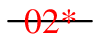} & \pic{.75}{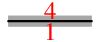} \\
  & \pic{.75}{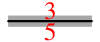} & \pic{.75}{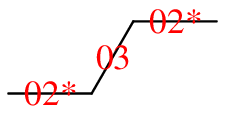} & \pic{.75}{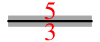} \\
  & \pic{.75}{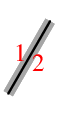} & \pic{.75}{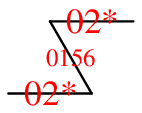} & \pic{.75}{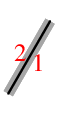} \\
\hline
F & \pic{.75}{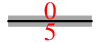} &                     & \pic{.75}{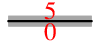} \\
  & \pic{.75}{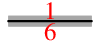} & \pic{.75}{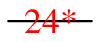}  & \pic{.75}{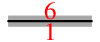} \\
  & \pic{.75}{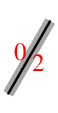} &                     & \pic{.75}{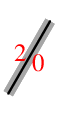} \\
\hline
\end{tabular}
\medskip
\caption{Gashes allowed in a gashed puzzle.}
\label{tab:gashes}
\end{table}

The bijection is formulated in terms of {\em gashed\/} single-row
puzzles in which some puzzle pieces next to each other do not have
matching labels.  More precisely, a {\em gash\/} is a pair of puzzle
edges with labels on both sides, together with a connected sequence of
edges between them, so that certain conditions are satisfied.  The two
edges with labels on both sides are called the {\em left leg\/} and
the {\em right leg\/} of the gash.  By definition, 
every gash must have one of the
types A, B, C, D, E, F, which correspond to the types of the Pieri
relations in Table~\ref{tab:pierirel}, but since the edges of a gash
need not all be horizontal, the definitions are not identical.
For each type there are a set
of choices for the left leg, the middle segment of edges, and the
right leg.  These choices are listed in Table~\ref{tab:gashes} and may
be combined in any way, as long as the orientation of the edges
remains as shown, and the height of the gash is at most
one, i.e.\ at most one non-horizontal edge may be included.  The
labels of the edges in Table~\ref{tab:gashes} should be understood in
the same way as for the Pieri relation in the previous section.  If a
horizontal edge is labeled with a sequence of numbers followed by a
star, then any number of connected horizontal edges with labels from
the sequence may be included.  A non-horizontal edge labeled with a
sequence of numbers means a single edge whose label is one of these
numbers.  Notice that Table~\ref{tab:pierirel} lists all possible
horizontal gashes.

\begin{example}
  Here are two gashes, one of type D and another of type F.\smallskip
  \[
  \pic{.75}{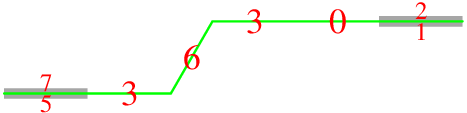}
  \raisebox{5mm}{\text{ \ \ \ \ and \ \ \ \ }}
  \pic{.75}{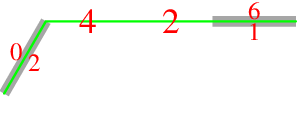}
  \]
\end{example}

Let $P$ be a single-row puzzle with border $(c_1,u',v',c_2)$ such that
$u \xra{1} u'$.  Then the label string $u'$ can be obtained by
interchanging two entries of $u$.  Change the corresponding two edges
on the top border of $P$ to have the entries of $u$ as their top
labels and the entries of $u'$ as their bottom labels.  This creates a
horizontal gash on the top border of $P$, and the resulting gashed
puzzle contains all the information required by the bijection.  We
will formulate the bijection as a transformation rule on gashed
puzzles.  This transformation takes a single-row puzzle with a gash on
the top border and changes it by {\em propagating\/} the gash to the
bottom border.

If a gash is not on the bottom border of its puzzle, then we define
the {\em front edge\/} of the gash as follows.  If the gash is
horizontal on the top border, then the front edge is the left leg.
Otherwise the front edge is the unique non-horizontal edge of the
gash.  A propagation is carried out by one or more steps that move the
front edge to the right.  The labels between the old and new front
edges may be changed in the process, and the type of the gash may
change as well.  The subset of puzzle pieces and edges that are
changed is called a {\em swap region}, and the change itself is called
a {\em swap}.  The result of the bijection is the gashed puzzle
obtained when the gash reaches the bottom border.  Before we give the
complete list of swap regions and the proof that propagations
are well-defined (in section~\ref{sec:propagation}), 
we first consider two examples.

\begin{example}
  Let $u = 0241$, $u' = 2041$, and $v' = 5410$.  If we start with the
  unique single-row puzzle with border $(2,u',v',0)$ and use $u$ to
  introduce a gash on the top border, then this gash is propagated to
  the bottom border by the following sequence of swaps.  Each of these
  swaps is carried out by applying a unique named swap
  region.\smallskip

  \noindent
  \mbox{}\hspace{5mm}
  \pic{.75}{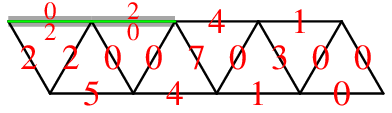} \ \raisebox{5mm}{$\mapsto$} \ 
  \pic{.75}{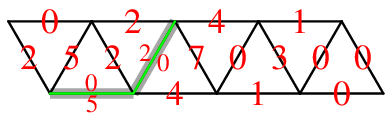} \ \raisebox{5mm}{$\mapsto$} \\ 
  \mbox{}\hspace{5mm}
  \pic{.75}{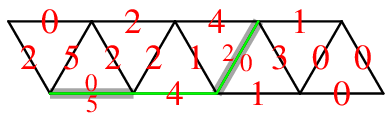} \ \raisebox{5mm}{$\mapsto$} \ 
  \pic{.75}{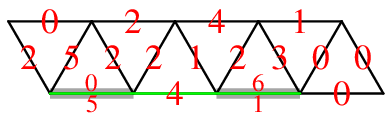} 
  \medskip
  
  \noindent
  The swap region that is applied first is called AF.  It has the
  following effect.
  \[
  \text{AF} \ : \ \raisebox{-6mm}{\pic{.75}{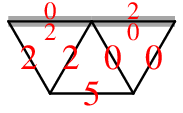}} \ \mapsto \ 
  \raisebox{-6mm}{\pic{.75}{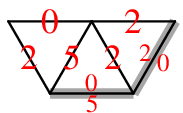}}
  \]
  The name indicates that a gash of type A is replaced with a gash of
  type F.  A more compact description of this rule is given in the
  following diagram.
  \[
  \text{AF} \ : \ \raisebox{-6mm}{\pic{.75}{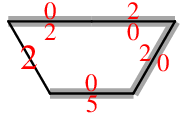}}
  \]
  This diagram shows the gashes both before and after the swap.  To
  obtain the region before the swap, one replaces all gashes on the
  bottom and right sides with their outside labels.  The region after
  the swap is obtained by replacing the gashes on the top and left
  sides with their outside labels.  In both cases the labels of the
  inner edges are uniquely determined by requiring that the interior
  of a swap region is a union of puzzle pieces with matching labels.
  The other two swap regions used in the example are called FF11 and
  FF9.
  \[
  \text{FF11} \ : \ \raisebox{-6mm}{\pic{.75}{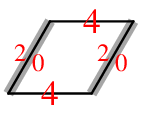}}
  \hspace{15mm} ; \hspace{15mm} \text{FF9} \ : \
  \raisebox{-6mm}{\pic{.75}{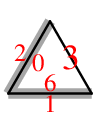}}
  \]
\end{example}

\begin{example}
  Let $u=1015$, $u'=1027$, and $v'=2031$, consider the unique
  single-row puzzle with border $(1,u',v',2)$, and use $u$ to create a
  gash on the top border.  In this case the propagation carries out
  the following sequence of swaps.\smallskip
  
  \noindent
  \mbox{}\hspace{5mm}
  \pic{.75}{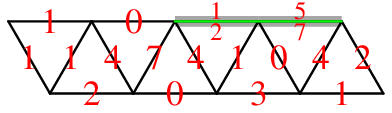} \ \raisebox{5mm}{$\mapsto$} \ 
  \pic{.75}{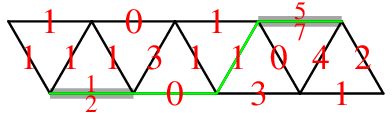} \ \raisebox{5mm}{$\mapsto$} \\ 
  \mbox{}\hspace{5mm}
  \pic{.75}{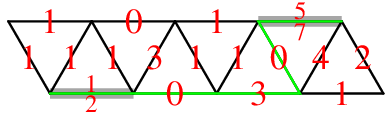} \ \raisebox{5mm}{$\mapsto$} \ 
  \pic{.75}{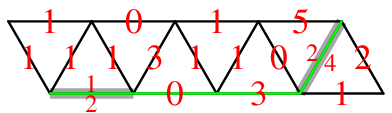} \ \raisebox{5mm}{$\mapsto$} \\ 
  \mbox{}\hspace{5mm}
  \pic{.75}{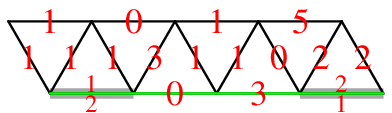}
  \medskip

  \noindent
  This example uses the swap regions DD2, DD11, DD17, and DD7.
  \[
  \begin{split}
    \text{DD2} \ : \ \raisebox{-6mm}{\pic{.75}{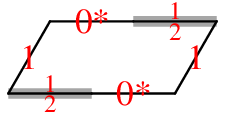}}
    \hspace{15mm} &; \hspace{15mm}
    \text{DD11} \ : \ \raisebox{-6mm}{\pic{.75}{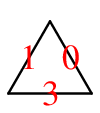}} \\
    \text{DD17} \ : \ \raisebox{-6mm}{\pic{.75}{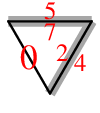}}
    \hspace{15mm} &; \hspace{15mm}
    \text{DD7} \ : \ \raisebox{-6mm}{\pic{.75}{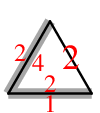}}
  \end{split}
  \]
  The label $0*$ on the first swap region DD2 indicates that the
  corresponding edges may be repeated any number of times, including
  zero.  The swap region DD11 does not cause any change to the puzzle;
  it simply allows the front edge of a gash to move to the right.
\end{example}

\section{Propagation rules}\label{sec:propagation}

In this section we give the complete list of swap regions required for
carrying out the bijection for $p=1$.  At the same time we prove that
any gash that is not on the bottom border of its puzzle can be moved
to the right by applying a unique swap region.  This establishes that
the list of swap regions gives a well defined map on gashed puzzles.

Each gash type comes with its own set of swap regions.  More
precisely, a swap region may be used only if its name starts with the
type of the gash at hand.  The proof that the list of swap regions is
complete consists of a case-by-case analysis of all possible gashes:
we exhaustively consider all cases for how the puzzle may look 
near the gash, and provide a unique swap region to cover every possibility.
This analysis is organized by gash type and comprises
sections~\ref{sec:srA} through \ref{sec:srF}.  The reader who does not wish
to verify the completeness of the analysis may safely skim these
sections.  In section~\ref{sec:properties}, we record the 
properties of the list that are essential to the proof of 
Proposition~\ref{prop:key}.

In the following we assume that we are given a gashed single-row
puzzle, such that the gash is not located on the bottom border.  We
will identify the unique swap region that must be applied to propagate
the gash.  If the front edge of the gash is not horizontal, then this
edge will be the left side of the swap region.  Similarly, if applying
the swap region results in a new gash that is not on the bottom
border, then the front edge of the new gash is taken to be the right
side of the swap region.  In all cases, 
the reader should
observe that the simplicity of the left and right border labels
implies that the indicated swap regions are completely contained in
the puzzle.

\def\srule#1{\global\def\srname{#1}\srname&\raisebox{-5mm}{\pic{.58}{swapreg\srname}}&\raisebox{-5mm}{\pic{.58}{swapregb\srname}}&\raisebox{-5mm}{\pic{.58}{swaprega\srname}}\\\hline}
  
\subsection{Swap regions for a gash of type A}\label{sec:srA}

Assume that the gash is of type A.  Then it is located on the top
border of the puzzle.  Let $a$ and $b$ be the labels of the edges
going south-west and south-east from the middle node of the gash.
\[
\psfrag{a}{\!\!\!\!$a$}
\psfrag{b}{\ \, $b$}
\pic{.75}{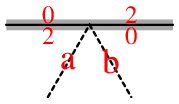}
\]
The following table lists all possible values of $a$ and $b$ together
with the unique swap region that can be applied in each case.
Notice that the `Before' and `After' fields indicate only one
particular instance of swap regions that include stretchable edges.

\begin{center}
  \begin{tabular}{|c|c|c|c|c|c|}
    \hline
    $a$ & $b$ & Name & Rule & Before & After \\
    \hline\hline
    0 & 0 & \srule{AA1}
    1 & 0 & \srule{AB}
    1 & 1 & \srule{AD}
    1 & 4 & \srule{AA2}
    2 & 0 & \srule{AF}
    2 & 1 & \srule{AC}
    2 & 2 & \srule{AA3}
    3 & 1 & \srule{AA4}
  \end{tabular}
\end{center}

\subsection{Swap regions for a gash of type B}\label{sec:srB}

Assume now that the gash is of type B, located on the top border of
the puzzle.  Let $a$ and $b$ be the labels of the edges indicated in
the picture.
\[
\psfrag{a}{\ \,$a$}
\psfrag{b}{\!\!$b$}
\pic{.75}{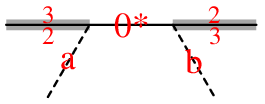}
\]
The table lists the possible values of $a$ and $b$ together with the
corresponding swap regions.
\begin{center}
  \begin{tabular}{|c|c|c|c|c|}
    \hline
    $a$ & $b$ & Name & Rule & Before/After \\
    \hline\hline
    0 & 0 & BA & \raisebox{-5mm}{\pic{.58}{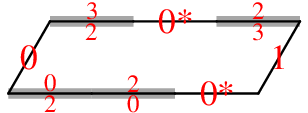}} &
    \raisebox{-6mm}{\shortstack{
      Before: \raisebox{-5mm}{\pic{.58}{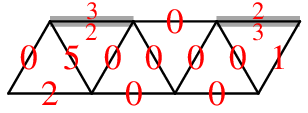}} \\
      After:\ \ \ \raisebox{-5mm}{\pic{.58}{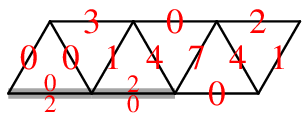}} 
    }}
    \\
    \hline
    1 & 0 & BB1 & \raisebox{-5mm}{\pic{.58}{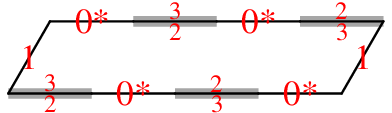}} &
    \raisebox{-6mm}{\shortstack{
      Before: \raisebox{-5mm}{\pic{.58}{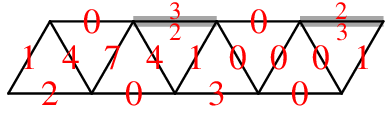}} \\
      After:\ \ \ \raisebox{-5mm}{\pic{.58}{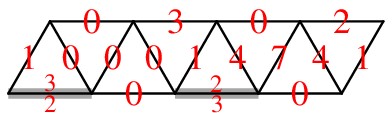}}
    }}
    \\
    \hline
    2 & 0 & BE & \raisebox{-5mm}{\pic{.58}{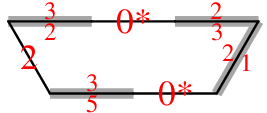}} &
    \raisebox{-6mm}{\shortstack{
      Before: \raisebox{-5mm}{\pic{.58}{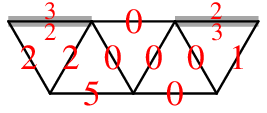}} \\
      After:\ \ \ \raisebox{-5mm}{\pic{.58}{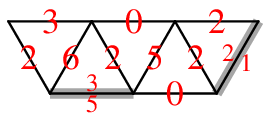}}
    }}
    \\
    \hline
    2 & 2 & BB2 & \raisebox{-5mm}{\pic{.58}{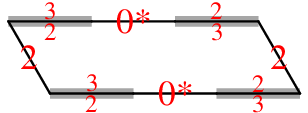}} &
    \raisebox{-6mm}{\shortstack{
      Before: \raisebox{-5mm}{\pic{.58}{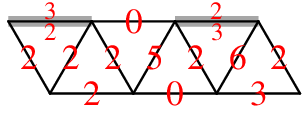}} \\
      After:\ \ \ \raisebox{-5mm}{\pic{.58}{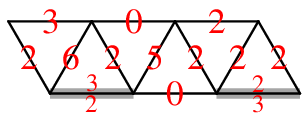}}
    }}
    \\
    \hline
  \end{tabular}
\end{center}

\subsection{Swap regions for a gash of type C}\label{sec:srC}

\subsubsection{}
Assume that the gash has type C and is located at the top border of the
puzzle.  Let $a$ be the indicated label.
\[
\psfrag{a}{\, $a$}
\pic{.75}{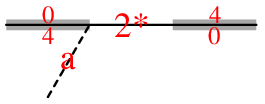}
\]
The table lists the possible values of $a$ and the corresponding swap
regions.
\begin{center}
  \begin{tabular}{|c|c|c|c|c|}
    \hline
    $a$ & Name & Rule & Before & After \\
    \hline\hline
    0 & \srule{CC1}
    \hline
    2 & \srule{CC2}
    \hline
  \end{tabular}
\end{center}

\subsubsection{}
Assume that the gash has type C with the following shape.  Let $a$ and
$b$ be the labels of the edges going south-east and east from the top
node of the left leg.  Notice that $b$ may be the right leg of the
gash, in which case the value of $b$ is displayed as $\frac{4}{0}$.
\[
\psfrag{a}{\ \ $a$}
\psfrag{b}{$b$}
\pic{.75}{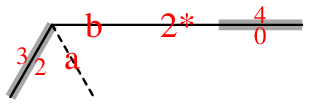}
\]
The table lists the possible values of $a$ and $b$ together with the
corresponding swap regions.
\begin{center}
  \begin{tabular}{|c|c|c|c|c|c|}
    \hline
    $a$ & $b$ & Name & Rule & Before & After \\
    \hline\hline
    0 & $\frac{4}{0}$ & \srule{CF}
    1 & $\frac{4}{0}$ & \srule{CC3}
    2 & $\frac{4}{0}$ & \srule{CA}
    2 & 2 & \srule{CC4}
  \end{tabular}
\end{center}

\subsubsection{}
Assume that the gash has type C with the following shape.  The unique
applicable swap region is determined by the indicated label $a$.
\[
\psfrag{a}{$a$}
\psfrag{b}{$b$}
\pic{.75}{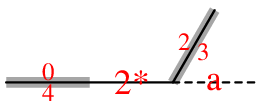}
\]
\begin{center}
  \begin{tabular}{|c|c|c|c|c|}
    \hline
    $a$ & Name & Rule & Before & After \\
    \hline\hline
    0 & \srule{CC5}
    2 & \srule{CC6}
  \end{tabular}
\end{center}

\subsection{Swap regions for a gash of type D}\label{sec:srD}

\subsubsection{}
Assume that the gash has type D and is located on the top border of
the puzzle.  Let $a$ be the labels of the left leg and let $b$ be the
label of the edge going south-west from the right node of the left
leg.
\[
\psfrag{a}{\raisebox{-1.5mm}{$a$}}
\psfrag{b}{\ \,$b$}
\pic{.75}{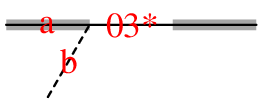}
\]
The table lists the possible values of $a$ and $b$ and the
corresponding swap regions.
\begin{center}
  \begin{tabular}{|c|c|c|c|c|c|}
    \hline
    $a$ & $b$ & Name & Rule & Before & After \\
    \hline\hline
    $\frac{1}{2}$ & 0 & \srule{DD1}
    $\frac{1}{2}$ & 1 & \srule{DD2}
    $\frac{1}{2}$ & 2 & \srule{DD3}
    $\frac{1}{2}$ & 3 & \srule{DA}
    $\frac{7}{5}$ & 2 & \srule{DD4}
  \end{tabular}
\end{center}

\subsubsection{}
Assume that the gash has type D with the following shape.
\[
\psfrag{a}{\ \ $a$}
\psfrag{b}{$b$}
\pic{.75}{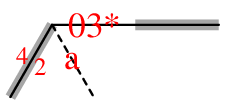}
\]
The unique applicable swap region is determined by the indicated label
$a$.
\begin{center}
  \begin{tabular}{|c|c|c|c|c|}
    \hline
    $a$ & Name & Rule & Before & After \\
    \hline\hline
    0 & \srule{DD5}
    1 & \srule{DE}
    2 & \srule{DD6}
    3 & \srule{DF}
  \end{tabular}
\end{center}

\subsubsection{}
Assume that the gash has type D with the following shape.
\[
\psfrag{a}{$a$}
\psfrag{b}{$b$}
\pic{.75}{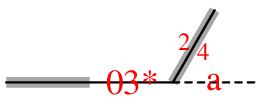}
\]
The unique applicable swap region is determined by the indicated label
$a$.
\begin{center}
  \begin{tabular}{|c|c|c|c|c|}
    \hline
    $a$ & Name & Rule & Before & After \\
    \hline\hline
    1 & \srule{DD7}
    7 & \srule{DD8}
  \end{tabular}
\end{center}

\subsubsection{}
Assume that the gash has type D with the following shape.
\[
\psfrag{a}{\!\!\!\!$a$}
\psfrag{b}{\ \ $b$}
\pic{.75}{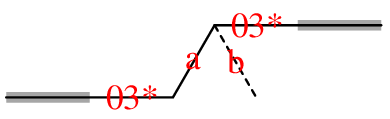}
\]
The unique applicable swap region is determined by the indicated labels
$a$ and $b$.
\begin{center}
  \begin{tabular}{|c|c|c|c|c|c|}
    \hline
    $a$ & $b$ & Name & Rule & Before & After \\
    \hline\hline
    0 & 0 & \srule{DD9}
    0 & 3 & \srule{DD10}
    1 & 0 & \srule{DD11}
    1 & 1 & \srule{DD12}
    5 & 2 & \srule{DD13}
    6 & 2 & \srule{DD14}
  \end{tabular}
\end{center}

\subsubsection{}
Assume that the gash has type D with the following shape (given by the
solid black lines).
\[
\psfrag{a}{\!\!\!$a$}
\psfrag{b}{\ \,$b$}
\pic{.75}{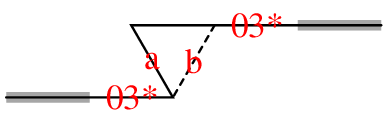}
\]
The unique applicable swap region is determined by the indicated labels
$a$ and $b$.
\begin{center}
  \begin{tabular}{|c|c|c|c|c|c|}
    \hline
    $a$ & $b$ & Name & Rule & Before & After \\
    \hline\hline
    0 & 0 & \srule{DD15}
    0 & 1 & \srule{DD16}
    0 & 4 & \srule{DD17}
    2 & 4 & \srule{DD18}
    2 & 5 & \srule{DD19}
    2 & 6 & \srule{DD20}
  \end{tabular}
\end{center}

\subsection{Swap regions for a gash of type E}\label{sec:srE}

\subsubsection{}
Assume that the gash has type E and is located on the top border of
the puzzle.  Let $a$ be the labels of the left leg and let $b$ be the
label of the edge going south-west from the right node of the left leg.
\[
\psfrag{a}{\raisebox{-1.5mm}{$a$}}
\psfrag{b}{\ \,$b$}
\pic{.75}{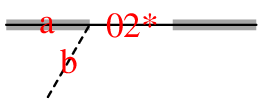}
\]
The table lists the possible values of $a$ and $b$ and the
corresponding swap regions.
\begin{center}
  \begin{tabular}{|c|c|c|c|c|c|}
    \hline
    $a$ & $b$ & Name & Rule & Before & After \\
    \hline\hline
    $\frac{1}{4}$ & 0 & \srule{EE1}
    $\frac{1}{4}$ & 2 & \srule{EE2}
    $\frac{3}{5}$ & 2 & \srule{EE3}
  \end{tabular}
\end{center}

\subsubsection{}
Assume that the gash has type E with the following shape.  Let $a$ be
the label of the edge going south-east from the top node of the left
leg, and let $b$ be the first non-zero label on the horizontal part of
the gash.  The following could be the labels of the right leg of the gash.
\[
\psfrag{a}{\ \ $a$}
\psfrag{b}{$b$}
\pic{.75}{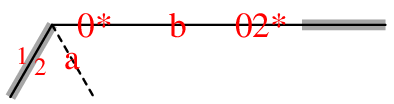}
\]
The table lists the possible values of $a$ and $b$ and the
corresponding swap regions.  In two cases the value of $b$ is omitted,
as it does not influence on the choice of swap region.
\begin{center}
  \begin{tabular}{|c|c|c|c|c|c|}
    \hline
    $a$ & $b$ & Name & Rule & Before & After \\
    \hline\hline
    0 & & \srule{EE4}
    1 & & \srule{EE5}
    2 & 2 & \srule{EE6}
    2 & $\frac{4}{1}$ & \srule{ED}
    2 & $\frac{5}{3}$ & \srule{EB}
    3 & $\frac{4}{1}$ & \srule{EF}
  \end{tabular}
\end{center}

\subsubsection{}
Assume that the gash has type E with the following shape.
\[
\psfrag{a}{$a$}
\psfrag{b}{$b$}
\pic{.75}{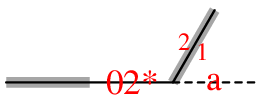}
\]
The unique applicable swap region is determined by the indicated label
$a$.
\begin{center}
  \begin{tabular}{|c|c|c|c|c|}
    \hline
    $a$ & Name & Rule & Before & After \\
    \hline\hline
    1 & \srule{EE7}
    2 & \srule{EE8}
    3 & \srule{EE9}
  \end{tabular}
\end{center}

\subsubsection{}
Assume that the gash has type E with the following shape.
\[
\psfrag{a}{\!\!\!\!$a$}
\psfrag{b}{\ \ $b$}
\pic{.75}{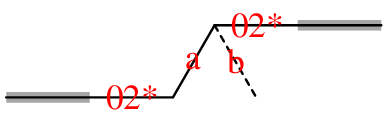}
\]
The unique applicable swap region is determined by the indicated labels
$a$ and $b$.
\begin{center}
  \begin{tabular}{|c|c|c|c|c|c|}
    \hline
    $a$ & $b$ & Name & Rule & Before & After \\
    \hline\hline
    0 & 0 & \srule{EE10}
    0 & 3 & \srule{EE11}
    0 & 5 & \srule{EE12}
    3 & 1 & \srule{EE13}
    3 & 6 & \srule{EE14}
  \end{tabular}
\end{center}

\subsubsection{}
Assume that the gash has type E with the following shape (given by the
solid black lines).
\[
\psfrag{a}{\!\!\!$a$}
\psfrag{b}{\ \,$b$}
\pic{.75}{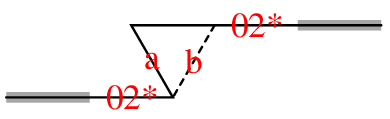}
\]
The unique applicable swap region is determined by the indicated labels
$a$ and $b$.
\begin{center}
  \begin{tabular}{|c|c|c|c|c|c|}
    \hline
    $a$ & $b$ & Name & Rule & Before & After \\
    \hline\hline
    0 & 0 & \srule{EE15}
    0 & 1 & \srule{EE16}
    1 & 1 & \srule{EE17}
    1 & 3 & \srule{EE18}
    5 & 0 & \srule{EE19}
    6 & 3 & \srule{EE20}
  \end{tabular}
\end{center}

\subsection{Swap regions for a gash of type F}\label{sec:srF}

\subsubsection{}
Assume that the gash has type F and is located on the top border.
Then the unique applicable swap region is determined by the labels $a$
of the left leg.
\[
\psfrag{a}{\raisebox{-1.5mm}{$a$}}
\pic{.75}{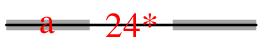}
\]
\begin{center}
  \begin{tabular}{|c|c|c|c|c|}
    \hline
    $a$ & Name & Rule & Before & After \\
    \hline\hline
    $\frac{0}{5}$ & \srule{FF1}
    $\frac{1}{6}$ & \srule{FF2}
  \end{tabular}
\end{center}

\subsubsection{}
Assume that the gash has type F with the following shape.  Let $a$ and
$b$ be the labels of the edges going south-east and east from the top
node of the left leg.  Notice that $b$ may be the labels of the right
leg.
\[
\psfrag{a}{\ \ $a$}
\psfrag{b}{$b$}
\pic{.75}{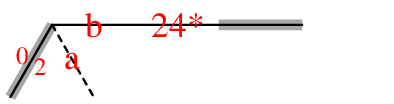}
\]
The table lists the possible values of $a$ and $b$ together with the
corresponding swap regions.
\begin{center}
  \begin{tabular}{|c|c|c|c|c|c|}
    \hline
    $a$ & $b$ & Name & Rule & Before & After \\
    \hline\hline
    0 & $\frac{5}{0}$ & \srule{FF3}
    1 & $\frac{5}{0}$ & \srule{FC}
    1 & $\frac{6}{1}$ & \srule{FE}
    1 & 4 & \srule{FF4}
    2 & $\frac{5}{0}$ & \srule{FA}
    2 & $\frac{6}{1}$ & \srule{FD}
    2 & 2 & \srule{FF6}
    3 & $\frac{6}{1}$ & \srule{FF7}
  \end{tabular}
\end{center}

\subsubsection{}
Assume that the gash has type F of the following shape.  Then the
unique applicable swap region is determined by the indicated label
$a$.
\[
\psfrag{a}{$a$}
\psfrag{b}{$b$}
\pic{.75}{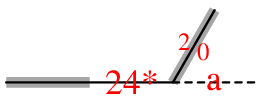}
\]
\begin{center}
  \begin{tabular}{|c|c|c|c|c|}
    \hline
    $a$ & Name & Rule & Before & After \\
    \hline\hline
    0 & \srule{FF8}
    1 & \srule{FF9}
    2 & \srule{FF10}
    4 & \srule{FF11}
  \end{tabular}
\end{center}

\subsection{Properties of the bijection}\label{sec:properties}

We finish this section by recording some consequences of the analysis
just carried out.  Given any single-row puzzle $P$ with a gash on its
top border, we let $\Phi(P)$ denote the puzzle obtained by propagating
the gash to the bottom border, using the swap regions of this section.
For any parallelogram shaped puzzle $P$, let $\rho(P)$ denote the 180 degree
rotation of $P$.

\begin{prop}\label{prop:invol1}
  The assignment $\Phi$ is a well defined map from the set of
  single-row puzzles with a gash on the top border into the set of
  single-row puzzles with a gash on the bottom border.  Furthermore,
  if $P$ is any single-row puzzle with a gash on the top border, then
  $\rho\, \Phi\, \rho\, \Phi(P) = P$.
\end{prop}
\begin{proof}
  The well definedness of $\Phi$ follows by observing that the 
  swap regions of sections \ref{sec:srA} through \ref{sec:srF} cover
  all possible cases.  For the second claim, suppose that $P$ is a
  gashed puzzle and $P'$ is the result of applying a swap region $\cR$
  to $P$.  An inspection of the swap region tables shows that also the
  180 degree rotation of $\cR$ is a swap region, and this swap region
  can be applied to $\rho(P')$ to produce $\rho(P)$.  If $\cR$ is
  called XY, where X and Y are distinct gash types, then the 180
  degree rotation of $\cR$ is called YX.  And if $\cR$ is called
  XX$m$, where X is a gash type and $m$ is an integer, then the
  rotation of $\cR$ is called XX$m'$ for a (possibly) different
  integer $m'$.  The proposition follows from this.
\end{proof}

\begin{example}
  The 180 degree rotation of the first propagation in
  section~\ref{sec:gashes} is carried out with the swap regions
  FF2, FF4, and FA.
  
  \noindent
  \mbox{}\hspace{5mm}
  \pic{.75}{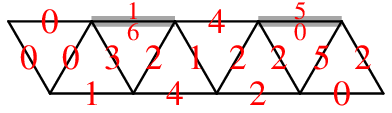} \ \raisebox{5mm}{$\mapsto$} \ 
  \pic{.75}{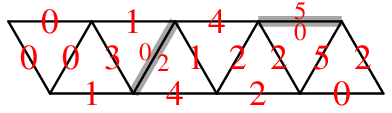} \ \raisebox{5mm}{$\mapsto$} \\ 
  \mbox{}\hspace{5mm}
  \pic{.75}{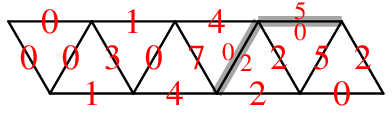} \ \raisebox{5mm}{$\mapsto$} \ 
  \pic{.75}{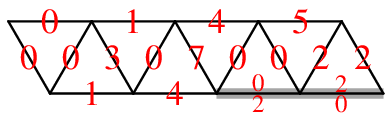} 
  \medskip
  
  The 180 degree rotation of the second propagation in
  section~\ref{sec:gashes} is carried out with the swap regions DD3,
  DD5, DD16, and DD12.
\end{example}

We record two additional consequences that are important for the
general bijection for $p \geq 2$.  Let $P$ be a gashed single-row
puzzle with label $c_1$ on the left border and label $c_2$ on the
right border.  We will say that $P$ has border
$(c_1,\frac{u}{u'},v',c_2)$ if $P$ has a horizontal gash on the top
border corresponding to the relation $u \xra{1} u'$, and the bottom
border of $P$ has labels $v'$.  Similarly we will say that $P$ has
border $(c_1,u,\frac{v}{v'},c_2)$ if $P$ has a horizontal gash on the
bottom border corresponding to the relation $v \xra{1} v'$, and the
top border of $P$ has labels $u$.

\begin{lemma}\label{lem:indexieq}
  Let $P$ be a single-row puzzle with a gash on its top border.  Let
  $(c_1,\frac{u}{u'},v',c_2)$ be the border of $P$, let
  $(c_1,u,\frac{v}{v'},c_2)$ be the border of $\Phi(P)$, let $(i,j)$
  be the index of $u \xra{1} u'$, and let $(k,l)$ be the index of $v
  \xra{1} v'$.  Then we have $i \leq l$ and $k < j$.  Moreover, if one
  of the relations $u\xra{1}u'$ or $v\xra{1}v'$ has type E, then $i-1
  \leq k$ and $j-1 \leq l$.
\end{lemma}
\begin{proof}
  The inequalities $i \leq l$ and $k < j$ are equivalent to the
  existence of a non-horizontal puzzle edge $e$, such that the top
  node of $e$ separates the left and right legs of the gash on $P$,
  and the bottom node of $e$ separates the left and right legs of the
  gash on $\Phi(P)$.  Assume at first that the propagation $P \mapsto
  \Phi(P)$ involves a swap region $\cR$ that moves both legs of the
  gash.  In this case an inspection of the propagation table shows
  that $e$ may be taken as one of the interior edges of $\cR$.
  Otherwise some intermediate puzzle in the propagation contains a
  non-horizontal gash whose front edge is different from both of its
  legs.  We may then take $e$ to be the front edge of this gash.  If
  $u\xra{1}u'$ has type E, then the second claim follows because none
  of the legs of a gash of type E are able to move more than one step
  to the left during a propagation.  This can be seen by inspecting
  the swap regions in section~\ref{sec:srE}.  Finally, if $v\xra{1}v'$
  has type E, then the same argument applies to $\rho\,\Phi(P)$.
\end{proof}

\begin{defn}
Let $P$ be a single-row puzzle with border $(c_1, \frac{u}{u'}, v',
c_2)$, and let $(c_1, u, \frac{v}{v'}, c_2)$ be the border of
$\Phi(P)$.  We will say that the propagation $P \mapsto \Phi(P)$ has
{\em type\/} X--Y if the relation $u \xra{1} u'$ has type X and the
relation $v \xra{1} v'$ has type Y.  Given a small triangle $\tau$ of
$P$, the {\em top part\/} of $\tau$ is the intersection of $\tau$ with
the top border of $P$, and the {\em bottom part\/} of $\tau$ is the
intersection of $\tau$ with the bottom border of $P$.  One of these
`parts' of $\tau$ is a point, and the other is a small horizontal line
segment.  We will say that $\tau$ is an {\em interior triangle\/} of
the propagation $P \mapsto \Phi(P)$ if the top part of $\tau$ is
located between the left and right legs of the gash on $P$, and the
bottom part of $\tau$ is located between the left and right legs of
the gash on $\Phi(P)$.
\end{defn}

\begin{lemma}\label{lem:interior}
  Let $P$ be a single-row puzzle with a horizontal gash on the top
  border, and assume that the propagation $P \mapsto \Phi(P)$ has type
  E--E.  Then all interior triangles are unchanged by this propagation
  and come from the list:

  \ \hspace{12mm}
  \pic{.5}{u000} \ \pic{.5}{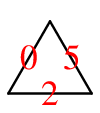} \ \pic{.5}{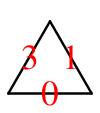} \ \pic{.5}{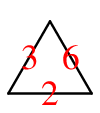} \ 
  \pic{.5}{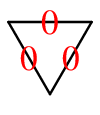} \ \pic{.5}{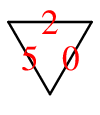} \ \pic{.5}{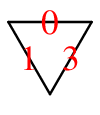} \ \pic{.5}{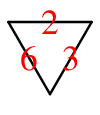}
\end{lemma}
\begin{proof}
  This follows by inspection of the swap regions in
  section~\ref{sec:srE}.  The triangles on the list correspond to the
  swap regions EE10, EE12, EE13, EE14, EE15, EE19, EE18, and EE20.
\end{proof}

\section{The general bijection}\label{sec:general}

Let $P$ be a single-row puzzle with border $(c_1,u',v',c_2)$ and let
$u$ be a label string such that $u \xra{p} u'$ for some $p$.  We
define a new single-row puzzle $\Phi^u(P)$ as follows.  If $u = u'$,
then set $\Phi^u(P) = P$.  If $p=1$ and $u \xra{1} u'$, then let $P'$
be the puzzle obtained from $P$ by changing the border to
$(c_1,\frac{u}{u'},v',c_2)$, let $(c_1,u,\frac{v}{v'},c_2)$ be the
border of $\Phi(P')$, and let $\Phi^u(P)$ be the puzzle obtained from
$\Phi(P')$ by changing this border to $(c_1,u,v,c_2)$.  Otherwise let
$u = u^0 \xra{1} u^1 \xra{1} \cdots \xra{1} u^p = u'$ be the unique
right-increasing Pieri chain from $u$ to $u'$.  By induction on $p$ we
may assume that $\Phi^{u^1}(P)$ has already been defined.  We then set
$\Phi^u(P) = \Phi^u(\Phi^{u^1}(P))$.

\ignore{
Let $P$ be a single-row puzzle with border $(c_1,u',v',c_2)$ and let
$u$ be a label string such that $u \xra{p} u'$ for some $p$.  We
define a new single-row puzzle $\Phi^u(P)$ as follows.  If $u = u'$,
then set $\Phi^u(P) = P$.  Otherwise let $u = u^0 \xra{1} u^1 \xra{1}
\cdots \xra{1} u^p = u'$ be the unique right-increasing Pieri chain
from $u$ to $u'$.  By induction on $p$ we may assume that $P_1 =
\Phi^{u^1}(P)$ has already been defined.  Let $(c_1,u^1,v^1,c_2)$ be
the border of $P_1$, and let $P_1^u$ be the puzzle obtained from $P_1$
by changing this border to $(c_1,\frac{u}{u^1},v^1,c_2)$.  Finally let
$(c_1,u,\frac{v}{v^1},c_2)$ be the border of $\Phi(P_1^u)$, and let
$\Phi^u(P)$ be the puzzle obtained from $\Phi(P_1^u)$ by changing the
border to $(c_1,u,v,c_2)$.  
}

\begin{example}\label{exm:rowex}
  Let $u = (1,0,2,4,2,5)$, $u' = (4,2,0,6,2,0)$, and $v' =
  (2,5,1,2,2,0)$.  We list the intermediate puzzles occurring when the
  map $\Phi^u$ is applied to a puzzle with border $(1,u',v',2)$.  For
  each step we have colored the union of the swap regions used to
  change the puzzle.
  \[
  \begin{split}
    & \raisebox{-6.5mm}{\pic{.75}{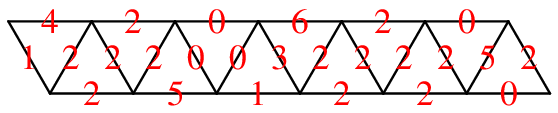}} \ \ \ \mapsto \\
    & \raisebox{-6.5mm}{\pic{.75}{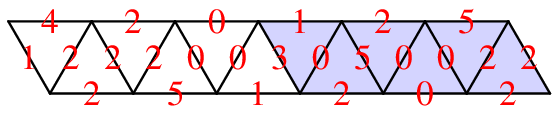}} \ \ \ \mapsto \\
    & \raisebox{-6.5mm}{\pic{.75}{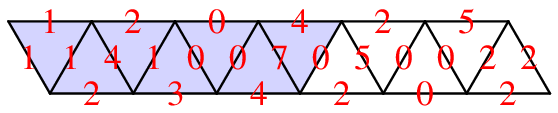}} \ \ \ \mapsto \\
    & \raisebox{-6.5mm}{\pic{.75}{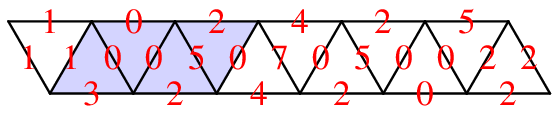}}
  \end{split}
  \]
\end{example}

\begin{lemma}\label{lem:factor}
  Let $\wt u$ be any label string contained in the unique
  right-increasing Pieri chain from $u$ to $u'$.  Then we have
  $\Phi^u(P) = \Phi^u(\Phi^{\wt u}(P))$.
\end{lemma}
\begin{proof}
  Let $u = u^0 \xra{1} u^1 \xra{1} \cdots \xra{p} u^p = u'$ be the
  right-increasing Pieri chain.  It follows from the definition that
  $\Phi^u(P) = \Phi^u(\Phi^{u^1}(P))$.  If $\wt u \neq u$, then $\wt
  u$ is contained in the unique right-increasing Pieri chain from
  $u^1$ to $u'$, so by induction on $p$ we obtain
  $\Phi^u(\Phi^{u^1}(P)) = \Phi^u(\Phi^{u^1}(\Phi^{\wt u}(P))) =
  \Phi^u(\Phi^{\wt u}(P))$, as required.
\end{proof}

\begin{lemma}\label{lem:chain}
  Let $P$ be a single-row puzzle with border $(c_1,u^p,v^p,c_2)$, and
  let $u^0 \xra{1} u^1 \xra{1} \cdots \xra{1} u^p$ be a
  right-increasing Pieri chain.  Let $(c_1,u^t,v^t,c_2)$ be the border
  of $\Phi^{u^t}(P)$ for each $t \in [0,p]$.  Then $v^0 \xra{1} v^1
  \xra{1} \cdots \xra{1} v^p$ is a Pieri chain.
\end{lemma}

Before we prove Lemma~\ref{lem:chain}, we emphasize that the produced
Pieri chain $v^0 \xra{1} v^1 \xra{1} \cdots \xra{1} v^p$ is not
necessarily right increasing;
when this occurs, we say that the propagation paths {\em cross}.
In addition, the lemma is false without
the assumption that the Pieri chain $u^0 \xra{1} u^1 \xra{1} \cdots
\xra{1} u^p$ is right increasing.  These points are essential to how
propagation paths are allowed to cross each other in a controlled way.

\begin{proof}
  Since $\Phi^{u^{t-1}}(P) = \Phi^{u^{t-1}}(\Phi^{u^t}(P))$ by
  Lemma~\ref{lem:factor}, it follows from the definition of $\Phi$ in
  section~\ref{sec:propagation} that $v^{t-1} \xra{1} v^t$ for each $t
  \in [1,p]$.  Let $(i_t,j_t)$ be the index of $u^{t-1} \xra{1} u^t$,
  and let $(k_t,l_t)$ be the index of $v^{t-1} \xra{1} v^t$ for each
  $t$.  Then we have $j_1 < j_2 < \dots < j_p$ by assumption, and
  Lemma~\ref{lem:indexieq} implies that $i_t \leq l_t$ and $k_t < j_t$
  for each $t$.  Let $1 \leq s < t \leq p$; we must show that $k_s <
  l_t$.  If $j_{t-1} \leq i_t$, then this is true because $k_s < j_s
  \leq j_{t-1} \leq i_t \leq l_t$.  Otherwise it follows from
  Lemma~\ref{lem:chainorder} that $u^{t-1} \xra{1} u^{t}$ has type E,
  so Lemma~\ref{lem:indexieq} implies that $j_t-1 \leq l_t$.  In this
  case we obtain $k_s < j_s \leq j_t-1 \leq l_t$, as required.
\end{proof}

\begin{cor}
  Let $P$ be a single-row puzzle with border $(c_1,u',v',c_2)$ such
  that $u \xra{p} u'$ for some $u$ and $p$, and let $(c_1,u,v,c_2)$ be
  the border of $\Phi^u(P)$.  Then $v \xra{p} v'$.
\end{cor}
\begin{proof}
  This follows from Lemma~\ref{lem:chain}.
\end{proof}

Notice that if we turn the last puzzle in Example~\ref{exm:rowex}
upside-down and apply $\Phi^{{v'}^\vee}$, then the sequence of
propagations in the example will be undone in reverse order.  
However, when the propagation paths cross, we get a slightly 
different propagation order.  This occurs when
a propagation of type A--A is carried out inside a propagation of
type E--E.

\begin{example}
  We show the steps involved in applying $\Phi^u$ to a single-row
  puzzle with border $(0,u',v',0)$, where $u = (3,0,0,2,2,4)$, $u' =
  (5,0,2,2,0,1)$, and $v' = u^\vee$.  The first step is a propagation
  of type E--E, while the second and third steps are propagations of
  type A--A.
  \[
  \begin{split}
    & \raisebox{-6.5mm}{\pic{.75}{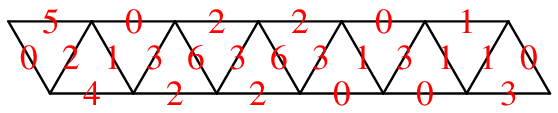}} \ \ \ \mapsto \\
    & \raisebox{-6.5mm}{\pic{.75}{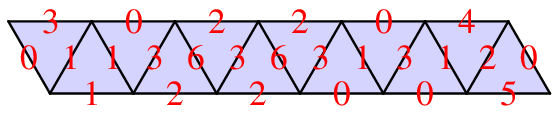}} \ \ \ \mapsto \\
    & \raisebox{-6.5mm}{\pic{.75}{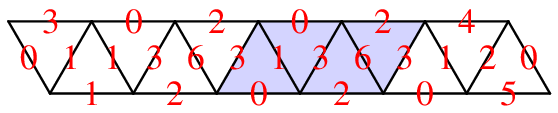}} \ \ \ \mapsto \\
    & \raisebox{-6.5mm}{\pic{.75}{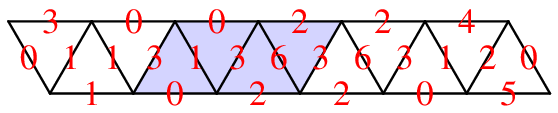}}
  \end{split}
  \]
  Notice that the resulting puzzle is equal to the 180 degree rotation
  of the initial puzzle $P$, so we have $\rho\, \Phi^u \rho\,
  \Phi^u(P) = P$.  However, the second application of $\Phi^u$ does
  not undo the propagations of the first application of $\Phi^u$ in
  the expected reverse order.  This is the main issue in the proof of
  Proposition~\ref{prop:inverse} below.
\end{example}

Let $Q$ be a single-row puzzle with border $(c_1,u,v,c_2)$ such that
$v \xra{p} v'$ for some $v'$ and $p$.  Then we define $\Phi_{v'}(Q) =
\rho\, \Phi^{{v'}^\vee}\!\rho(Q)$.  Lemma~\ref{lem:chain} implies that
this puzzle has border $(c_1,u',v',c_2)$ for a label string $u'$ with
$u \xra{p} u'$.

\begin{prop}\label{prop:inverse}
  Let $P$ be a single-row puzzle with border $(c_1,u',v',c_2)$ and let
  $u$ be a label string such that $u \xra{p} u'$ for some $p$.  Then
  we have $\Phi_{v'}(\Phi^u(P)) = P$.
\end{prop}
\begin{proof}
  We proceed by induction on $p$.  The statement is clear if $p=0$,
  and for $p=1$ it follows from Proposition~\ref{prop:invol1}.  Assume
  that $p \geq 2$ and let $u = u^0 \xra{1} u^1 \xra{1} \cdots \xra{1}
  u^p = u'$ be the unique right-increasing Pieri chain.  For each $t
  \in [0,p]$ we set $P_t = \Phi^{u^t}(P)$, and we let
  $(c_1,u^t,v^t,c_2)$ be the border of this puzzle.  Then $v = v^0
  \xra{1} v^1 \xra{1} \cdots \xra{1} v^p = v'$ is a Pieri chain by
  Lemma~\ref{lem:chain}.  Let $(i_t,j_t)$ be the index of $u^{t-1}
  \xra{1} u^t$ and let $(k_t,l_t)$ be the index of $v^{t-1} \xra{1}
  v^t$.

  Assume first that $k_p = \max\{k_1,\dots,k_p\}$.  Then it follows
  from Proposition~\ref{prop:rightincr}(b) and Lemma~\ref{lem:factor}
  that $\Phi_{v'}(P_0) = \Phi_{v'}(\Phi_{v^{p-1}}(P_0))$.  Since $P_0
  = \Phi^u(P_{p-1})$, we obtain from the induction hypothesis that
  $\Phi_{v^{p-1}}(P_0) = P_{p-1}$.  We deduce that
  $\Phi_{v'}(\Phi^u(P)) = \Phi_{v'}(P_0) =
  \Phi_{v'}(\Phi_{v^{p-1}}(P_0)) = \Phi_{v'}(P_{p-1}) = P$, as
  required.

  Otherwise we have $k_p < k_s$ for some $s \leq p-1$.  Since $\{ v^t
  \}$ is a Pieri chain, this implies that $k_p+1 \leq k_s <
  \min(l_{p-1},l_p)$.  It follows that the relation $v^{p-1} \xra{1}
  v^p$ is not of type A, and Lemma~\ref{lem:chainorder} implies that
  $k_p < k_{p-1} < l_{p-1} < l_p$, the relation $v^{p-2} \xra{1}
  v^{p-1}$ has type A, and $v^{p-1} \xra{1} v^p$ has type E.
  Lemma~\ref{lem:indexieq} now implies that $i_p \leq k_p+1 \leq
  k_{p-1} < j_{p-1}$, so another application of
  Lemma~\ref{lem:chainorder} shows that $i_p < i_{p-1} < j_{p-1} <
  j_p$, the relation $u^{p-2} \xra{1} u^{p-1}$ has type A, and the
  relation $u^{p-1} \xra{1} u^p$ has type E.  Notice also that
  $k_t \leq l_{p-1}-1 = k_{p-1}$ for each $t \in [1,p-1]$, so
  $k_{p-1} = \max\{k_1,\dots,k_p\}$.

  The propagation $P = P_p \mapsto P_{p-1}$ is carried out by first
  changing the border of $P_p$ to $(c_1, \frac{u^{p-1}}{u^p}, v^p,
  c_2)$, then applying $\Phi$, and finally changing the border of the
  resulting puzzle to $(c_1,u^{p-1},v^{p-1},c_2)$.  This application
  of $\Phi$ is therefore a propagation of type E--E.  Furthermore, the
  step $P_{p-1} \mapsto P_{p-2}$ is carried out by changing the border
  of $P_{p-1}$ to $(c_1, \frac{u^{p-2}}{u^{p-1}}, v^{p-1}, c_2)$,
  applying $\Phi$, and changing the border of the result to
  $(c_1,u^{p-2},v^{p-2},c_2)$.  Here the application of $\Phi$ is a
  propagation of type A--A, which can happen only by applying a single
  swap region $\cR$ of type AA1, AA2, AA3, or AA4.  Notice that all
  small triangles of $\cR$ (before the swap) must be interior
  triangles to the propagation $P_p \mapsto P_{p-1}$ of type E--E.  We
  therefore deduce from Lemma~\ref{lem:interior} that $\cR$ must have
  type AA1 or AA4.  Since the triangles after a swap of type AA1 or
  AA4 are also on the list of Lemma~\ref{lem:interior}, the E--E and
  A--A propagations can be carried out in the opposite order with the
  same result.

  More precisely, let $\wt u$ be the unique label string described in
  Lemma~\ref{lem:chainorder}, so that $u^{p-2} \xra{1} \wt u \xra{1}
  u'$ is a Pieri chain, $u^{p-2} \xra{1} \wt u$ has type E, and $\wt u
  \xra{1} u'$ has type A.  Similarly, let $\wt v$ be the unique label
  string such that $v^{p-2} \xra{1} \wt v \xra{1} v'$ is a Pieri
  chain, $v^{p-2} \xra{1} \wt v$ has type E, and $\wt v \xra{1} v'$
  has type A.  Now set $\wt P = \Phi^{\wt u}(P)$.  Then $\wt P$ has
  border $(c_1,\wt u,\wt v,c_2)$ and $\Phi^{u^{p-2}}(\wt P) =
  P_{p-2}$.

  Using that $u^0 \xra{1} \cdots \xra{1} u^{p-2} \xra{1} \wt u$ is a
  right-increasing Pieri chain, we obtain $P_0 = \Phi^u(P_{p-2}) =
  \Phi^u(\Phi^{u^{p-2}}(\wt P)) = \Phi^u(\wt P)$.  
  Since $v^0 \xra{1}
  \cdots \xra{1} v^{p-2} \xra{1} \wt v \xra{1} v'$ is a Pieri chain and
  $\wt v \xra{1} v'$ has index $(k_{p-1},l_{p-1})$ 
  with $k_{k-1} = \max\{k_1,\dots,k_p\}$, we obtain from
  Proposition~\ref{prop:rightincr}(b) and Lemma~\ref{lem:factor} that
  $\Phi_{v'}(P_0) = \Phi_{v'}(\Phi_{\wt v}(P_0))$. 
  Finally, since the induction hypothesis implies that $\Phi_{\wt
    v}(P_0) = \Phi_{\wt v}(\Phi^u(\wt P)) = \wt P$, we obtain
  $\Phi_{v'}(\Phi^u(P)) = \Phi_{v'}(P_0) = \Phi_{v'}(\Phi_{\wt
    v}(P_0)) = \Phi_{v'}(\wt P) = P$, as required.
\end{proof}

\begin{proof}[Proof of Proposition~\ref{prop:key}]
  The identity (\ref{eqn:id}) follows from Corollary~\ref{cor:idpuz}.
  Let $u$, $v'$, $w_1$, and $w_2$ be 012-strings and let $p \in \N$.
  We will say that a parallelogram shaped puzzle has border $(w_1,u,v',w_2)$
  if $u$ gives the labels of the top border, $v'$ gives the labels of the
  bottom border, $w_1$ gives the labels of the left border, and $w_2$ gives
  the labels of the right border, all in north-west to south-east
  order.  
  Recall from section~\ref{sec:multone} that
  to prove the identity (\ref{eqn:assoc}) it suffices to
  construct a bijection between the set of parallelogram shaped puzzles with
  border $(w_1,u',v',w_2)$ such that $u \xra{p} u'$, and the set of
  parallelogram shaped puzzles with border $(w_1,u,v,w_2)$ such that $v
  \xra{p} v'$.  We do this by modifying one row at the time.
  \[
  {
    \psfrag{l}{$\!\!\!\!\!\!\!\!\!w_1$}
    \psfrag{r}{$\ \ \,w_2$}
    \psfrag{t}{\raisebox{3mm}{$\!\!\!u'$}}
    \psfrag{b}{\raisebox{-3mm}{$v'$}}
    \pic{.6}{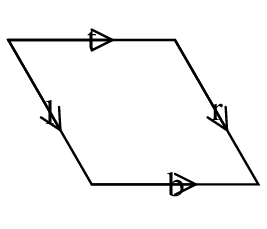}
  }
  \raisebox{11mm}{\ \ \ \ \ \ $\longleftrightarrow$\ \ \ \ \ \ }
  {
    \psfrag{l}{$\!\!\!\!\!\!\!\!\!w_1$}
    \psfrag{r}{$\ \ \,w_2$}
    \psfrag{t}{\raisebox{3mm}{$\!\!\!u$}}
    \psfrag{b}{\raisebox{-3mm}{$v$}}
    \pic{.6}{rhompuz2}
  }
  \]

  Given a parallelogram shaped puzzle $P'$ with border $(w_1,u',v',w_2)$,
  let $n$ be the number of rows in this puzzle, let $P'_i$ be the
  subpuzzle in the $i$-th row for $1 \leq i \leq n$ (counted from top
  to bottom), and let
  $(c_i,{u'}^{i-1},{u'}^i,c'_i)$ be the border of $P'_i$.  Then
  ${u'}^0 = u'$ and ${u'}^n = v'$.  Set $u^0 = u$ and $P_1 =
  \Phi^{u^0}(P'_1)$.  Assume inductively that $P_1,\dots,P_i$ have
  already been defined.  Then let $(c_i,u^{i-1},u^i,c'_i)$ be the
  border of $P_i$ and set $P_{i+1} = \Phi^{u^i}(P'_{i+1})$.  Finally,
  let $\Phi^u(P')$ be the union of the rows $P_i$ for $i \in [1,n]$,
  and let $v$ be the labels of the bottom border of this puzzle.  Then
  $\Phi^u(P')$ is a valid puzzle with border $(w_1,u,v,w_2)$ such that
  $v \xra{p} v'$.  Finally, it follows from
  Proposition~\ref{prop:inverse} that $\rho\, \Phi^{{v'}^\vee} \!
  \rho\, \Phi^u (P') = P'$, which implies that the map $P' \mapsto
  \Phi^u(P')$ is a bijection.
\end{proof}

\begin{example}\label{example:cross}
  Here is an example of the bijection in the proof of
  Proposition~\ref{prop:key} in a case where two propagation paths
  cross each other.
  \begin{align*}
    & \pic{.7}{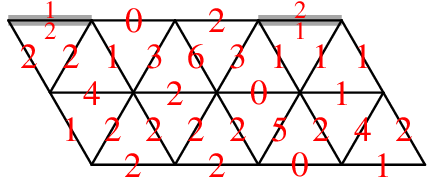} \raisebox{11mm}{$\mapsto$}
      \pic{.7}{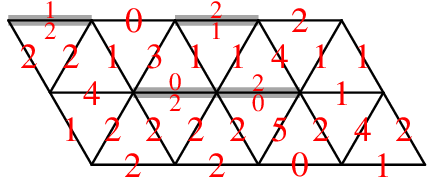} \raisebox{11mm}{$\mapsto$} \\[-5.5pt]
    & \pic{.7}{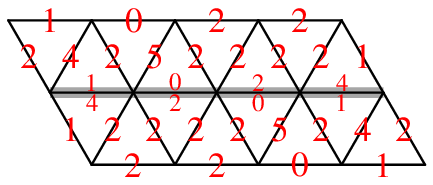} \raisebox{11mm}{$\mapsto$}
      \pic{.7}{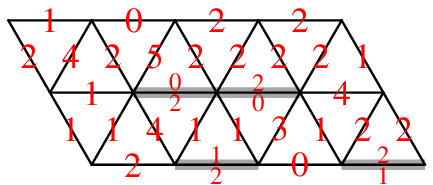} \raisebox{11mm}{$\mapsto$} \\[-5.5pt]
    & \pic{.7}{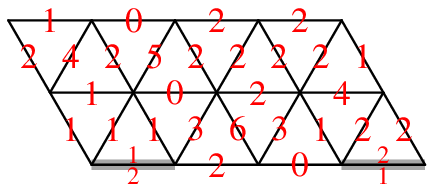}
  \end{align*}

  Here is what happens if the propagations are carried out one after
  another: that is, each gash is allowed to propagate to the 
  bottom of the puzzle before the next propagation begins.  
  \begingroup
  \allowdisplaybreaks[1]
  \begin{align*}
    & \pic{.7}{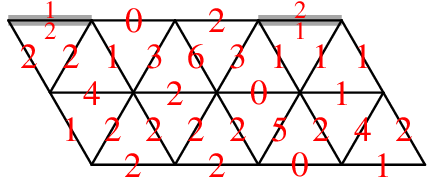} \raisebox{11mm}{$\mapsto$}
      \pic{.7}{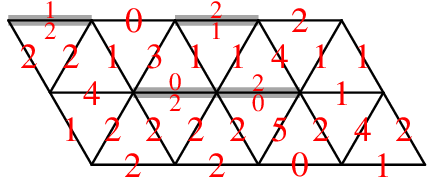} \raisebox{11mm}{$\mapsto$} \\[-5.5pt]
    & \pic{.7}{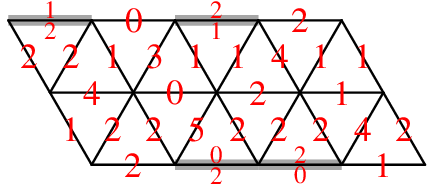} \raisebox{11mm}{$\mapsto$}
      \pic{.7}{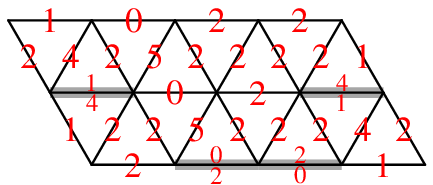} \raisebox{11mm}{$\mapsto$} \\[-5.5pt]
    & \pic{.7}{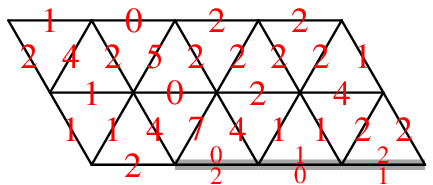}
  \end{align*}
  \endgroup
  Notice that the resulting 012-strings $v = (2,0,1,2)$ and
  $v' = (2,2,0,1)$ on the bottom border do not satisfy $v \xra{2} v'$.
  This illustrates why several propagations must be handled 
  simultaneously and in the correct sequence, in order to obtain a 
  proof of Proposition~\ref{prop:key}.
\end{example}

\section{A quantum Littlewood-Richardson rule}
\label{sec:qlrrule}

Let $Y={\mathrm G}(m,n)$ denote the Grassmannian parametrizing
$m$-dimensional complex linear subspaces of ${\mathbb C}^n$. The
Schubert varieties $Y_u$ in $Y$ and their classes $[Y_u]$ in
$H^*(Y,\Z)$ may be indexed by 02-strings $u=(u_1,\ldots, u_n)$ with $m$
zeroes and $n-m$ twos. The codimension of $Y_u$ in $Y$ is equal to the
number of inversions $\ell(u)$.

Let $u$, $v$, and $w$ be three 02-strings as above and fix a
nonnegative integer $d$ such that $\ell(u)+\ell(v)+\ell(w)=
m(n-m)+nd$.  
The three-point, genus zero Gromov-Witten invariant
$\langle [Y_u], [Y_v], [Y_w] \rangle_d$ 
may be defined as the number of
rational maps $f\colon \bP^1\to Y$ of degree $d$ such that $f(0)\in
Y_u$, $f(1)\in Y_v$, and $f(\infty)\in Y_w$, whenever the Schubert
varieties $Y_u$, $Y_v$, and $Y_w$ are taken to be in general
position. When $d=0$, we have that $\langle [Y_u], [Y_v], [Y_w]
\rangle_0$ is equal to the triple intersection number $\int_Y[Y_u]\cdot
[Y_v]\cdot [Y_w]$, which is a Schubert structure constant in the
cohomology of the Grassmannian $Y$, and given by the classical
Littlewood-Richardson rule. 
In general, 
the invariants $\langle
[Y_u], [Y_v], [Y_w] \rangle_d$ are Schubert structure constants
in the small quantum cohomology ring of $Y$, which is a
$q$-deformation of $H^*(Y,\Z)$.

If $d \leq \min(m,n-m)$, define a $012$-string $u_d$ by changing 
the first $d$ twos and the
last $d$ zeroes of $u$ to ones. For example, if $Y={\mathrm G}(4,10)$,
$d=2$, and $u=2022020202$, then $u_d=1012021212$. We similarly define
the $012$-strings $v_d$ and $w_d$. Our main theorem may now be used to
establish the following conjecture of Buch, Kresch, and Tamvakis from
\cite[section 2.4]{buch.kresch.tamvakis:gromov-witten}.

\begin{mainthm}[Quantum Littlewood-Richardson Rule]
Let $Y_u$, $Y_v$, and $Y_w$ be Schubert varieties in $\mathrm{G}(m,n)$, 
and suppose 
that
$\ell(u)+\ell(v)+\ell(w)= m(n-m)+nd$.
The Gromov-Witten invariant
$\langle [Y_u], [Y_v], [Y_w] \rangle_d$ is equal to the number of
triangular puzzles for which $u_d$, $v_d$, and $w_d$ are the labels
on the left, right, and bottom sides, in clockwise order, when
$d \leq \min(m,n-m)$, and is zero otherwise.
\end{mainthm}
\begin{proof}
The $012$-strings $u_d$, $v_d$, and $w_d$ index Schubert varieties
$X_{u_d}$, $X_{v_d}$, and $X_{w_d}$ in the two-step flag variety
$\mathrm{Fl}(m-d,m+d;n)$. According to 
\cite[Corollary 1]{buch.kresch.tamvakis:gromov-witten}
we have
\[
\langle [Y_u], [Y_v], [Y_w] \rangle_d =
\int_{\mathrm{Fl}(m-d,m+d;n)} [X_{u_d}] \cdot [X_{v_d}] \cdot [X_{w_d}].
\]
The desired result follows by applying Theorem~\ref{thm:main}.
\end{proof}


\providecommand{\bysame}{\leavevmode\hbox to3em{\hrulefill}\thinspace}
\providecommand{\MR}{\relax\ifhmode\unskip\space\fi MR }
\providecommand{\MRhref}[2]{%
  \href{http://www.ams.org/mathscinet-getitem?mr=#1}{#2}
}
\providecommand{\href}[2]{#2}

\end{document}